\documentclass[11pt]{amsart}

\usepackage[a4paper,hmargin=3.5cm,vmargin=4cm]{geometry}
\usepackage{amsfonts,amssymb,amscd,amstext}
\usepackage{graphicx}
\usepackage[dvips]{epsfig}
\usepackage{quoting}
\usepackage{hyperref}

\usepackage{fancyhdr}
\input xy
\xyoption{all}
\usepackage{tikz}
\usetikzlibrary{matrix}

\usepackage[color=blue!20]{todonotes}

\usepackage{enumitem}
\usepackage{titlesec}
\usepackage{mathrsfs}

\titleformat{\section}
{\filcenter\bfseries} {\thesection{.}}{0.2cm}{}

\titleformat{\subsection}[runin]
{\bfseries} {\thesubsection{.}}{0.15cm}{}[.]

\titleformat{\subsubsection}[runin]
{\em}{\thesubsubsection{.}}{0.15cm}{}[.]

\usepackage[up,bf]{caption}

\newtheorem{theorem}{Theorem}[section]

\newtheorem{lemma}[theorem]{Lemma}
\newtheorem{claim}[theorem]{Claim}
\newtheorem{corollary}[theorem]{Corollary}

\theoremstyle{definition}
\newtheorem{definition}[theorem]{Definition}
\newtheorem{remark}[theorem]{Remark}

\numberwithin{equation}{section}
\numberwithin{figure}{section}

\newcommand\Ccal{\mathcal{C}}

\newcommand\Hcal{\mathcal{H}}

\newcommand\Mcal{\mathcal{M}}

\newcommand\Tcal{\mathcal{T}}

\newcommand\Ascr{\mathscr{A}}

\newcommand\Cscr{\mathscr{C}}

\newcommand\Oscr{\mathscr{O}}

\newcommand\C{\mathbb{C}}
\newcommand\D{\overline{\mathbb D}}
\newcommand\CP{\mathbb{CP}}
\renewcommand\D{\mathbb D}

\newcommand\N{\mathbb{N}}
\renewcommand\P{\mathbb{P}}

\renewcommand\S{\mathbb{S}}

\newcommand\Z{\mathbb{Z}}

\renewcommand\c{\mathbb{C}}

\newcommand\cp{\mathbb{CP}}

\newcommand\h{\mathbb{H}}
\newcommand\n{\mathbb{N}}
\renewcommand\r{\mathbb{R}}
\newcommand\s{\mathbb{S}}

\newcommand\z{\mathbb{Z}}

\newcommand\igot{\mathfrak{i}}

\renewcommand\igot{\mathfrak{i}}

\newcommand\Agot{\mathfrak{A}}

\newcommand\wt{\widetilde}
\newcommand\wh{\widehat}

\newcommand\dist{\mathrm{dist}}

\newcommand\CMC{\mathrm{CMC\text{-}1}}

\newcommand{\SL}{{\rm{SL}}_2(\C)}

\newcommand{\ord}{{\rm ord}}

\newcommand\Qcal{\mathcal{Q}}

\def\dist{\mathrm{dist}}

\usepackage{color}
\usepackage{array}

\begin{document}
\fancyhead[LO]{Holomorphic null curves in the special linear group}
\fancyhead[RE]{A.\ Alarc\'on and J.\ Hidalgo}
\fancyhead[RO,LE]{\thepage}

\thispagestyle{empty}

\begin{center}
{\bf\Large Holomorphic null curves in the special linear group
} 

\medskip
%

{\bf Antonio Alarc\'on and
Jorge Hidalgo}
\end{center}

\medskip

\begin{quoting}[leftmargin={5mm}]
{\small
\noindent {\bf Abstract}\hspace*{0.1cm}
In this paper we 
develop the theory of approximation for holomorphic null curves in the special linear group $\SL$. In particular, we establish Runge, Mergelyan, Mittag-Leffler, and Carleman type theorems for the family of holomorphic null immersions $M\to\SL$ from any open Riemann surface $M$. Our results include jet interpolation of Weierstrass type and approximation by embeddings, as well as global conditions on the approximating curves. 
As application, we show that every open Riemann surface admits a proper holomorphic null embedding into $\SL$, and hence also a proper conformal immersion of constant mean curvature $1$ into hyperbolic 3-space. This settles a problem posed by Alarc\' on and Forstneri\v c in 2015.

\noindent{\bf Keywords}\hspace*{0.1cm} 
Special linear group, Riemann surface, holomorphic null curve, holomorphic approximation, constant mean curvature surface, Bryant surface, $\CMC$ face.

\noindent{\bf Mathematics Subject Classification (2020)}\hspace*{0.1cm} 
32H02,
32E30,
53A10, 
53C42. 
}
\end{quoting}


\section{Introduction}\label{sec:intro}
\noindent
A holomorphic immersion $F:M\to\SL\subset\C^4$ from an open Riemann surface $M$ into the
special linear group
\[
\SL = \left\{  z=   \left( \begin{array}{cc}
z_{11} & z_{12}\\
z_{21}& z_{22} 
\end{array} \right) \in \C^4 \colon  \det z =   z_{11} z_{22} - z_{12} z_{21}  = 1 \right\}
\]
is said to be {\em null}, or a  {\em holomorphic null curve}, if the pull-back by $F$ of the Killing form on the Lie algebra $\mathfrak{sl}_2 (\c)$ vanishes identically, that is, the $\mathfrak{sl}_2 (\c)$-valued holomorphic $1$-form $F^{-1}dF$ satisfies
\[
\det(F^{-1}dF)=0\quad \text{at every point of $M$.}
\]
The term {\em null} was introduced by Bryant in
\cite[p.\ 328]{Bryant1987Asterisque} and refers to the fact that $F$ is null if and only if tangent vectors to $F(M)$ are null (or lightlike) for the Cartan-Killing metric on $\SL$; see e.g.\ \cite[Sec.\ 11.5]{JensenMussoNicolodi2016} for further discussion. 
Holomorphic null curves in $\SL$ are interesting on several grounds. A main reason why these curves have received attention in recent decades is that they project to the real hyperbolic 3-space $\h^3 =\SL / {\rm SU}_2$ as conformal immersions of constant mean curvature $1$ ($\CMC$) \cite[Theorem A]{Bryant1987Asterisque}, also called {\em Bryant surfaces}. 

The aim of this paper is to initiate and develop the theory of approximation for holomorphic null curves $M\to\SL$ from any open Riemann surface $M$. In particular, we shall establish analogues in this framework to the classical Runge, Mergelyan, Mittag-Leffler, and Carleman approximation theorems in complex analysis (see e.g.\ \cite{FFW18}). Our results also include jet interpolation of Weierstrass type and approximation by embeddings, as well as global conditions on the approximating curves, such as being (topologically) proper or complete (in the Riemannian sense). 

Here is a sampling of the kind of results that are obtained in this paper; see Theorem \ref{th:mt} and Corollary \ref{cr:Carleman} for much more precise statements.
%
%
\begin{theorem}\label{th:sampling}
Let $M$ be an open Riemann surface, $K\subset M$ a Runge compact subset, and $\Lambda$ and $E$ a pair of disjoint finite (possibly empty) subsets of 
$K$, and assume that $F:K\setminus E\to\SL$ is a holomorphic null immersion such that $F|_\Lambda$ is injective and $F$ extends meromorphically to $K$ with an effective pole at each point in $E$. Then, for any $\varepsilon>0$ and any integer $k>0$ there is a proper holomorphic null embedding $\wt F:M\setminus E\to\SL$ such that $\wt F-F$ is holomorphic on $K$, $|\wt F-F|<\varepsilon$ everywhere on $K$, and $\wt F-F$ vanishes to order $k$ at every point in $\Lambda\cup E$.
\end{theorem}

The results in this paper serve as a flexible tool for constructing examples of holomorphic null curves in $\SL$ and $\CMC$ surfaces in $\h^3$ with prescribed complex structure.
Some well-known widely used tools to construct and analyse these objects are the Bryant representation \cite{Bryant1987Asterisque} and the Small formula \cite{Small1994} (see also \cite{deLimaRoitman2002,KokubuUmeharaYamada2003,SaEarpToubiana2004}). Furthermore, these surfaces, which are locally isometric to minimal surfaces in Euclidean space $\r^3$ by the Lawson correspondence \cite{Lawson1970AM} (see also \cite{GalvezMira2008}), admit a spinor representation \cite{BobenkoPavlyukevichSpringborn2003} (see also \cite{MussoNicolodi2009}) and have been studied by the method of moving frames (see e.g.\   \cite[Ch.\ 11]{JensenMussoNicolodi2016} and \cite{Pirola2007AJM}). 
A first immediate consequence of Theorem \ref{th:sampling} worth noting is the following existence result for properly embedded holomorphic null curves in $\SL$ with arbitrary complex structure. It settles the problem posed by Alarc\' on and Forstneri\v c in \cite[Problem 1, p.\ 919]{AlarconForstneric2015MA}, which was the initial motivation for the investigation in this paper.
%
%
\begin{corollary}\label{co:properSL2C}
Every open Riemann surface admits a proper holomorphic null embedding into $\SL$.
\end{corollary}

Set $\C^*=\C\setminus\{0\}$, as customary. Our method of proof relies on the fact pointed out by Mart\'in, Umehara, and Yamada in \cite[Sec.\ 3.1]{MartinUmeharaYamada2009CVPDE} that the biholomorphism $\Tcal: \C^2\times\C^*\to \SL\setminus \{z_{11}=0\}$ given by 
\begin{equation}\label{eq:t}
\Tcal(z_1,z_2,z_3)=\frac{1}{z_3}\left(\begin{matrix}
1 & z_1+\igot z_2\\
z_1-\igot z_2 & z_1^2+z_2^2+z_3^2
\end{matrix}\right),\quad (z_1,z_2,z_3)\in \C^2\times\C^*,
\end{equation}
takes holomorphic null curves in $ \C^2\times\C^*$ (see Section \ref{sec:lindf} for background) into holomorphic null curves in $\SL\setminus\{z_{11}=0\}$. The inverse holomorphic map $\Tcal^{-1}:\SL\setminus \{z_{11}=0\}\to \C^2\times\C^*$ (see \eqref{eq:t-1})
takes holomorphic null curves in $\SL\setminus\{z_{11}=0\}$ into holomorphic null curves in $ \C^2\times\C^*$ as well. Note that the biholomorphism $\Tcal$ is a restriction of a birational map between $\C^3$ and $\SL$. This fact, which has not been exploited until now in the study of null curves, will be a key novelty in our method of proof; more on this in what follows.

The transformation $\Tcal$ was already used by Alarc\'on and Forstneri\v c  in order to show that every bordered Riemann surface\footnote{A {\em bordered Riemann surface} $R$ is an open connected Riemann surface which is the interior $R=\overline R\setminus b\overline R$ of a compact one dimensional complex manifold $\overline R$ with smooth boundary $b\overline R\neq \varnothing$ consisting of finitely many closed Jordan curves.} $R$ admits a proper holomorphic null embedding $R\to\SL$ (with the range in $\SL\setminus\{z_{11}=0\}$) \cite[Corollary 2]{AlarconForstneric2015MA}; see 
\cite{Castro-InfantesHidalgo2024} for further results in this direction. The construction in \cite{AlarconForstneric2015MA} consists of finding a holomorphic null embedding $X:R\to\C^2\times\C^*$ with a so particular asymptotic behavior that $\Tcal\circ X:R\to\SL$ is a proper map. On the other hand,  the theory of interpolation for holomorphic null curves in $\SL$ was recently initiated in \cite{AlarconCastro-InfantesHidalgo2023}, where the authors and Castro-Infantes established a Weierstrass-Florack type theorem to the effect that given an open Riemann surface $M$, a closed discrete subset $\Lambda\subset M$, and a holomorphic null curve $F$ from a neighborhood of $\Lambda$ into $\SL\setminus\{z_{11}=0\}$, there is a complete holomorphic null immersion $\wt F:M\to\SL$ (with the range in $\SL\setminus\{z_{11}=0\}$) such that $\wt F-F$ vanishes to any given finite order $k(p)$ at each point $p\in \Lambda$. The proof in \cite{AlarconCastro-InfantesHidalgo2023} consists of using approximation and interpolation theory for holomorphic null curves in $\C^3$ (see \cite{AlarconForstnericLopez2021Book}) in order to interpolate $\Tcal^{-1}\circ F$ on $\Lambda$ by a holomorphic null immersion $X:M\to\C^2\times\C^*$, thereby ensuring that $\wt F=\Tcal\circ X:M\to \SL\setminus\{z_{11}=0\}$ interpolates $F$ on $\Lambda$ as well. As in the case of the properness in the aforementioned result from \cite{AlarconForstneric2015MA}, completeness of $\Tcal\circ X$ is ensured in this process by granting some extra fancy conditions on the  holomorphic null curve $X$ in $\C^2\times\C^*$. We refer to \cite{AlarconLopez2013MA,AlarconForstneric2015MA} for further results where the biholomorphism $\Tcal$ has played a key role.

Despite having shown to be fruitful, this approach to the construction of holomorphic null curves in $\SL$ has strong limitations, since one is forced to stay in $\SL\setminus\{z_{11}=0\}$ and $\C^2\times\C^*$. In particular, the question in  \cite[Problem 1]{AlarconForstneric2015MA} (now settled by Corollary \ref{co:properSL2C}) about the existence of proper examples remained open. In this paper we keep using the transformation $\Tcal$ back and forth in a fundamental way, but we proceed in a more accurate manner that enables us to use the whole space $\SL$ and even to deal with singular meromorphic null curves with prescribed singular locus and polar set (see Remark \ref{rk:E}). This allows, as a first outcome, to develop a genuine theory of approximation for holomorphic null curves in $\SL$, which is the main purpose of this paper. Then, with this in hand, adding global properties such as completeness or properness to the approximating curves becomes a fairly standard task (see Section \ref{sec:lproper} for further discussion).

Let us conclude this introduction by pointing out a couple more applications of Theorem \ref{th:sampling}. First, concerning Bryant surfaces, since ${\rm SU}_2$ is compact, the canonical projection $\SL\to\h^3=\SL/{\rm SU}_2$ is a proper map, so it takes proper holomorphic null curves in $\SL$ into proper conformal $\CMC$ immersions. 
Thus, the following 
result is a consequence of Theorem \ref{th:sampling} and the fact that every simply connected Bryant surface lifts to a holomorphic null curve in $\SL$ \cite{Bryant1987Asterisque}. 
%
%
\begin{corollary}[Weak Runge Theorem for Bryant surfaces]
\label{co:bryant}
Let $M$ be an open Riemann surface, $K\subset M$ a simply connected Runge compact subset, and $\Lambda\subset K$ a finite subset. Then, every conformal $\CMC$ immersion $\varphi:K\to\h^3$ can be approximated uniformly on $K$ by proper conformal $\CMC$ immersions $M\to\h^3$ agreeing with $\varphi$ to any given finite order on $\Lambda$.

In particular, every open Riemann surface is the underlying complex structure of a proper $\CMC$ surface in $\h^3$.
\end{corollary}
The second assertion in Corollary \ref{co:bryant} settles the first (less ambitious) question in \cite[Problem 1]{AlarconForstneric2015MA}. The compact set $K\subset M$ here (as well as that in Theorem \ref{th:sampling}) need not be connected. We do not know whether the assumption that $K$ is simply connected can be removed from the statement of the corollary; this is why this result is weaker than a honest Runge type theorem.  Nevertheless, our more general results in Section \ref{sec:ms} lead to more precise existence, approximation, and interpolation theorems for $\CMC$ surfaces in $\h^3$ with arbitrary conformal structure and control on the asymptotic behavior; see Theorem \ref{th:Bryant}.

On the other hand, holomorphic null curves in $\SL$ project also to de Sitter 3-space $\S_1^3=\SL/{\rm SU}_{1,1}$ as $\CMC$ {\em faces} (that is, spacelike conformally immersed surfaces of constant mean curvature 1 with admissible singularities in $\S_1^3$); see \cite[Def.\ 1.4 and Theorem 1.9]{Fujimori06}, \cite[Def.\ 1.1]{FKKRUY13}, and also Section \ref{sec:bryant}, for more details. Conversely, every simply connected $\CMC$ face in $\S_1^3$ lifts to a holomorphic null curve in $\SL$. Moreover, complete holomorphic null curves in $\SL$ project to {\em weakly complete} $\CMC$ faces in $\S_1^3$ according to \cite[Def.\ 1.3]{FRUYY09}.  Since proper null curves in $\SL$ are complete, the following easily follows from Theorem \ref{th:sampling}; see Theorem \ref{th:deSitter} for a more precise statement.
%
%
\begin{corollary}[Weak Runge Theorem for  $\CMC$ faces]
\label{co:deSitter}
If $M$, $K$, and $\Lambda$ are as in Corollary \ref{co:bryant}, then every $\CMC$ face $\phi:K\to\s_1^3$ can be approximated uniformly on $K$ by weakly complete $\CMC$ faces $M\to\s_1^3$ agreeing with $\phi$ to any given finite order on $\Lambda$.
\end{corollary}
Since the Lie group ${\rm SU}_{1,1}\subset \SL$ (see \eqref{eq:SU11}) is not compact, the projection $\SL\to\s_1^3=\SL/{\rm SU}_{1,1}$ is not a proper map, and it remains an open question whether every open Riemann surface $M$ admits a proper $\CMC$ face $M\to\s_1^3$; this is actually unknown even for bordered Riemann surfaces. Partial results regarding this question were obtained in \cite{AlarconCastro-InfantesHidalgo2023,Castro-InfantesHidalgo2024}.

%
%
\subsubsection*{Organization of the paper} 
In Section \ref{sec:lemmas} we introduce preliminaries and state three key lemmas to be used in the proof of the main results of the paper. These lemmas are proved in Sections \ref{sec:lindf} to \ref{sec:c3}.  In Section \ref{sec:ms} we state our main results on approximation for holomorphic null curves in $\SL$, and prove them via the lemmas in Section \ref{sec:lemmas}. Finally, Section \ref{sec:bryant} is devoted to Bryant surfaces and $\CMC$ faces. 

%
%
\section{Preliminaries and key lemmas}\label{sec:lemmas}
%
%
\subsection{Preliminaries and notation}\label{subsec:notation}

We write $\n = \{1, 2,3,\ldots\}$ and $\Z_+ = \n\cup\{0\}$. We shall denote $\igot = \sqrt{-1}\in\C$, by $| \cdot |$ and $\dist ( \cdot, \cdot)$ the Euclidean norm and distance on $\C^n$ $(n \in \n)$, respectively, and by $\| \cdot \|_{\infty}$ the infinity norm on $\C^n$.
Given two maps $f,g \colon Y \to Z$ between sets, we write $f \equiv g$ to denote that $f(y) = g(y)$ for all $y \in Y$; we write $f \not \equiv g$ otherwise. 

Throughout the paper and whenever convenient, we shall identify the space $\Mcal_2 (\C)$ of $2\times 2$ complex matrices and $\C^4$ by
\begin{equation}\label{eq:c4}
\Mcal_2 (\C) \ni 	z = \left( \begin{array}{cc}
z_{11} & z_{12}\\
z_{21}& z_{22} 
\end{array} \right) \longleftrightarrow z = (z_{11}, z_{12}, z_{21}, z_{22}) \in \C^4.
\end{equation} 
Note that a holomorphic immersion $F:M\to\SL\subset\C^4$ from an open Riemann surface $M$ into $\SL$
is null if and only if 
it is {\em directed} by the quadric variety
\begin{equation}\label{eq:nullquadricsl2}
\Agot =\left\lbrace z = \left(\begin{matrix}
z_{11}& z_{12}\\
z_{21} & z_{22}
\end{matrix}\right) \in \C^4 \colon \det z = z_{11} z_{22} - z_{12} z_{21}=0 \right\rbrace,
\end{equation}
meaning that the derivative $F'\colon M\to\C^4$ of $F$ with respect to any local holomorphic coordinate on $M$ assumes its values in $\Agot_* :  = \Agot \setminus\{0\}$, that is, $\det (F')\equiv 0$. 
%
%
\begin{definition}[\text{\cite[Def.\ 1.12.9]{AlarconForstnericLopez2021Book}}]
\label{def:admissible}
An \emph{admissible set} in a smooth surface $M$ is a compact set of the form $S = K \cup \Gamma \subset M $, where $K$ is a (possibly empty) finite union of pairwise disjoint compact domains with smooth boundaries in $M$, and $\Gamma= S \setminus \mathring K =\overline{S\setminus K}$ is a (possibly empty) finite union of pairwise disjoint smooth Jordan arcs and closed Jordan curves meeting $K$ only at their endpoints (if at all) and such that their intersections with the boundary $bK$ of $K$ are transverse. 
\end{definition}

Let $M$ be an open Riemann surface,  $N$ a complex manifold, and $\varnothing\neq S \subset M$ a subset.
We denote by $\Cscr^r(S,N)$, $r\in \Z_+$,  the space of all maps $S\to N$ of class $\Cscr^r$ on a neighborhood of $S$ in $M$ and write $\Cscr^r (S) = \Cscr^r (S, \C)$.
If $f \in \Cscr^r (S)^n$ ($n \in \N$) and $S$ is compact, we denote by $\| f \|_{r,S}$ the $\Cscr^r$ maximum norm of $f$ in $S$; if $S$ is not compact then we set $\| f \|_{r,S}:=\sup\{\|f\|_{r,K}\colon K\subset S$, $K$ compact$\}$.
We write $\Oscr (S, N)$ for the space of all holomorphic maps from some neighborhood of $S$ in $M$ into $N$. We set $\Oscr(S) = \Oscr(S, \C)$, whereas $\Oscr_{\infty} (S)$ will denote the space of meromorphic functions on some neighborhood of $S$ in $M$. Note that $\Oscr (S) \subset \Oscr_{\infty} (S)\subset \Oscr(S,\cp^1)$.  
Given a closed discrete subset $E \subset \mathring{S}$, we set 
\[
\Oscr_{\infty} ( S | E) = \Oscr_{\infty} (S) \cap \Oscr (S\setminus E). 
\]
Assume in the sequel that $S\subset M$
is a locally connected closed set in $M$ all whose connected components are admissible sets (Definition \ref{def:admissible}).
We set
\[
\Ascr^r (S, N) = \Cscr^r (S, N) \cap \Oscr(\mathring{S}, N) ,  \quad \Ascr^r (S)= \Ascr^r(S, \C),
\]
and
\[
\Ascr_{\infty}^r (S) = \Cscr^r (S, \C \P^1) \cap \Cscr^r (b S, \C) \cap \Oscr_{\infty} (\mathring{S}).
\]
By the identity principle, 
a function $f \in \Ascr^r_{\infty}(S)$ has at most finitely many poles on each component of $S$, all of which lie in $\mathring{S}$.
Given a 
set $E\subset\mathring S$ having a finite intersection with each component of $S$, we denote
\[
\Ascr^r_{\infty} (S| E) = \Ascr^r_{\infty} (S) \cap \Oscr (\mathring{S} \setminus E).
\]
For a map $f = (f_1, \hdots, f_n ) \in \Cscr^0 (S, \C \P^1 )^n$ and $n \in \N$, we write 
\begin{equation}\label{eq:f-1}
f^{-1} ( \infty) = \bigcup_{j=1}^n f_j^{-1} (\infty).
\end{equation}
If in addition $N\subset \C^n$ $(n \in \N)$ is a complex submanifold, then we also set:
\begin{itemize}
\item $\Oscr_{\infty} (S, N) = \{ f \in \Oscr_{\infty} (S)^n \colon f( S \setminus f^{-1} (\infty) ) \subset N \}$.
\item $\Ascr^r_{\infty} (S, N) = \{f \in \Ascr^r_{\infty} (S)^n \colon f (S \setminus f^{-1} (\infty)) \subset N  \}$.
\item $\Oscr_{\infty} (S | E, N) = \Oscr_{\infty} (S, N)  \cap \Oscr ( S \setminus E, N)$.
\item $\Ascr^r_{\infty} (S | E , N) = \Ascr^r_{\infty} (S, N) \cap \Oscr ( \mathring{S} \setminus E, N)$.
\end{itemize}

Recall that a compact subset $K$ of an open Riemann surface $M$ is said to be \emph{Runge} (also {\em holomorphically convex} or {\em $\Oscr(M)$-convex}) if every function in $\Ascr (K)=\Cscr^0(K)\cap\Oscr(\mathring K)$ may be approximated uniformly on $K$ by functions in $\Oscr(M)$. By the Runge-Mergelyan theorem  \cite[Theorem 5]{FFW18} this is equivalent to $M\setminus K$ having no relatively compact connected components in $M$. 
In order to state a Mergelyan type theorem for holomorphic null curves in $\SL$, we need the following notion.
%
%
\begin{definition}
\label{def:gmncsl}
Let $S = K \cup \Gamma$ be an admissible set (Definition \ref{def:admissible}) in an open Riemann surface $M$, $\theta$  a nowhere vanishing holomorphic $1$-form on a neighborhood of $S$, and $E\subset \mathring{K}=\mathring{S}$ a finite subset. A {\em generalized null curve $S\setminus E \to \SL$ of class $\Ascr^r_{\infty} (S|E)$ $(r \in \n)$} is a pair $(F ,f \theta)$ with $F \in \Ascr^r_{\infty} (S | E,  \SL)$  and $f \in \Ascr^{r-1}_{\infty} ( S|E, \mathbf{\Agot}_*)$ (see \eqref{eq:nullquadricsl2}) satisfying the following.
\begin{itemize}
\item $f \theta = dF$ holds on $K \setminus E$ (so $F \colon \mathring{K} \setminus E \to \SL$ is a holomorphic null curve). 
\item For a smooth path $\alpha$ in $M$ parameterizing a connected component of $\Gamma$ we
have $ \alpha^* (f \theta) = \alpha^* (d F)= d ( F \circ \alpha)$.
\item $f^{-1} (\infty) \subset E$ (see \eqref{eq:f-1}).
\end{itemize}
This notion naturally extends to the case when $S\subset M$ is a locally connected closed set all whose connected components are admissible sets, by asking the pair $(F ,f \theta)$ to be a generalized null curve on each component of $S$. 
\end{definition}
We shall say that a holomorphic null curve which is also a topological embedding is a {\em holomorphic null embedding}. Given a subset $W \subset M$,
we will say that a map $W \to \SL$ is a {\it holomorphic null curve} if it extends to a holomorphic null curve on some open neighborhood of $W$ in $M$, and that a map $W \to \C^n$ is {\em flat} if its image is contained in an affine complex line in $\C^n$, and {\em nonflat} otherwise. A map $W \to \SL$ is nonflat if it is nonflat as a map into $\C^4$; see \eqref{eq:c4}.

%
%
\subsection{Key lemmas} 

The following three lemmas, whose proofs are deferred to Sections \ref{sec:lindf} to \ref{sec:c3}, are the fundamental new tools that we exploit in the proof of the main result in this paper; see Theorem \ref{th:mt} in Section \ref{sec:ms}.
%
%
\begin{lemma}[Semiglobal approximation] \label{lemma:indf} 
Let $M$ be an open Riemann surface, $\theta$ a nowhere vanishing holomorphic $1$-form on $M$, $S = K \cup \Gamma \subset M$ a Runge admissible subset (see Definition \ref{def:admissible}), $\Lambda$ and $E$ disjoint finite subsets of $\mathring{S}$, $L \subset M$ a smoothly bounded compact domain with $S  \subset \mathring L$, and $(F , f \theta)$ a generalized null curve $S\setminus E \to \SL$ of class $\Ascr_{\infty}^r (S | E)$ $(r \in \N)$
(see Definition \ref{def:gmncsl}).
Then, given $\varepsilon>0$ and $k\in\n$, there exists a nonflat holomorphic null curve $\wt F: L\setminus E \to \SL$ of class $\Oscr_\infty(L|E, \SL)$ satisfying the following.
\begin{enumerate}[label = \rm (\roman*)]
\item \label{f-}$\wt F  - F$ is continuous on $S$, hence of class $\Ascr^r(S)^4$. 
\item \label{fap}  $\| \wt F  - F \|_{r, S} < \varepsilon$.
\item \label{fint}$\wt F - F$ vanishes to order $k$ at every point of $\Lambda \cup E$. 
\newcounter{l23'}\setcounter{l23'}{\value{enumi}}
\end{enumerate}
\end{lemma}
%
%
\begin{lemma}[Embeddedness]\label{lemma:emb}
Let $M$ be an open Riemann surface, $K \subset M$ a smoothly bounded compact domain, $\Lambda$ and $E$ disjoint finite subsets of $\mathring{K}$, and $F : K \setminus E \to \SL$ a nonflat holomorphic null curve of class $\Oscr_\infty(K|E, \SL)$.
 If $F|_{\Lambda}$ is injective, then for any $\varepsilon>0$,  $k\in \N$, and $r \in \Z_+$ there is a nonflat holomorphic null embedding $\wt F : K \setminus E \to \SL$ of class $\Oscr_\infty(K|E, \SL)$ satisfying the following conditions.
\begin{enumerate}[label = \rm (\roman*)]
\item \label{f-em}$\wt F  - F$ is continuous on $K$, hence holomorphic.
\item \label{fapem}  $\| \wt F  - F \|_{r, K} < \varepsilon$.
\item \label{fintem}$\wt F - F$ vanishes to order $k$ at every point of $\Lambda \cup E$. 
\end{enumerate}
\end{lemma}
%
%
\begin{lemma}[Properness]\label{lemma:proper}
Let $M$ be an open Riemann surface, $K$ and $L$ smoothly bounded Runge compact domains in $M$ with $K\subset  \mathring{L}$, $\Lambda$ and $E$ disjoint finite subsets of $\mathring{K}$, and $F \colon K \setminus E \to \SL$ a nonflat holomorphic null curve of class $\Oscr_\infty(K|E, \SL)$.
For any $\varepsilon>0$,  $k\in\n$, $r \in \Z_+$, and
\[
	0 < \delta < \min \{ \| F (p) \|_{\infty}  \colon p \in b K \},
\]
there exists a nonflat holomorphic null curve $\wt{F}  \colon L\setminus E \to \SL$ of class $\Oscr_\infty(L|E, \SL)$ such that:
\begin{enumerate}[label= \rm (\roman*)]
\item \label{condaf} $\wt F - F$ is continuous on $K$, hence holomorphic.
\item \label{fpap} $\| \wt{F} - F \|_{r,K} < \varepsilon$.
\item \label{condbf} $\wt{F} - F$ vanishes to order $k$ at every point of $\Lambda \cup E$. 
\item \label{delta}  $\| \wt F (p) \|_{\infty} > \delta$ for all $p \in L \setminus \mathring{K}$.
\item \label{1eps}  $\| \wt F (p) \|_{\infty} > 1 / \varepsilon$ for all $p \in bL$. 
\end{enumerate}
\end{lemma}
Note that the given curve $F$ in Lemmas \ref{lemma:indf}, \ref{lemma:emb}, and \ref{lemma:proper} need not have effective poles at the points in $E$, that is, it could extend holomorphically to some of them. 

Let us also record here the following elementary observation.
%
%
\begin{remark}\label{rk:flambda}
If $M$ is an open Riemann surface, $F = (F_1, F_2, F_3, F_4) \colon M \to \SL$ is a holomorphic null curve (see \eqref{eq:c4}), and $\lambda \in \C$, then the maps 
\[
\left(\begin{matrix}
F_1  + \lambda F_3 & F_2 + \lambda F_4 \\
F_3  & F_4
\end{matrix}\right) , \quad 
\left(\begin{matrix}
F_1  & F_2\\
F_3 + \lambda F_1& F_4 + \lambda F_2 
\end{matrix}\right), 
\]
\[
\left(\begin{matrix}
F_1  + \lambda F_2 & F_2\\
F_3 + \lambda F_4 & F_4
\end{matrix}\right), \quad \text{and}\quad
\left(\begin{matrix}
F_1   & F_2 + \lambda F_1 \\
F_3  & F_4 + \lambda F_3
\end{matrix}\right)
\]
are holomorphic null curves  $M \to \SL$.
It is clear that if $F^\lambda$ is any of these maps and $p, q \in M$, then $F(p) = F(q)$ if and only if $F^\lambda (p) = F^\lambda (q)$. 
These statements also hold for meromorphic null curves and for generalized null curves in $\SL$. 
\end{remark}

%
%

\section{Approximation by holomorphic null curves in $\SL$}\label{sec:ms}
\noindent
In this section we present the main results in the paper, giving Theorem \ref{th:sampling} as a special case. We begin with the following analogues for holomorphic null curves in $\SL$ of the classical Runge-Behnke-Stein-Mergelyan-Bishop approximation theorem for holomorphic functions on open Riemann surfaces \cite{Runge1885,BehnkeStein1949,Mergelyan1951,Bishop1958PJM} combined with the Mittag-Leffler and the Weierstrass-Florack interpolation theorems \cite{Mittag-Leffler1884AM,Weierstrass1885, Florack1948}. We refer to \cite{FFW18} for a survey on holomorphic approximation. 
%
%
\begin{theorem}\label{th:mt}
Let  $M$ be an open Riemann surface, $\theta$ a nowhere vanishing holomorphic $1$-form on $M$, $S \subset M$ a locally connected closed subset all whose connected components are Runge admissible compact sets, and $\Lambda$ and $E$ disjoint closed discrete subsets of $M$ with $\Lambda \cup E \subset \mathring S$, and assume that  $(F , f \theta)$ is a generalized null curve
$S\setminus E\to\SL$ of class $\Ascr^r_{\infty}(S|E)$ $(r \in \N)$.
Then, for any $\varepsilon>0$ and any map $k\colon \Lambda \cup E \to \Z_+$ there is a nonflat holomorphic null curve $\wt F : M \setminus E \to \SL$ of class $\Oscr_\infty(M|E, \SL)$ satisfying the following conditions.
\begin{enumerate}[label = \rm (\alph*)]
\item \label{mt-}  $\wt F - F$ is continuous on $S$, hence of class $\Ascr^r(S)^4$. 
\item \label{mtapp}$\| \wt F - F \|_{r, S} < \varepsilon$.
\item \label{mtint} $\wt F - F$ vanishes to order $k(p)$ at every point $p\in \Lambda \cup E$.
\item  \label{mtiny} If $F|_\Lambda$ is injective, then we can choose $\wt F :  M \setminus E \to \SL$ to be injective.
\item \label{mtap} If $F$ has an effective pole at every point in $E$, then we can choose $\wt F : M \setminus E \to \SL$ to be an almost proper map\footnote{A map between topological spaces is {\em almost proper} if connected components of preimages of compact sets are compact.}, hence a complete immersion.
\item  \label{mtp} If $F : S \setminus E \to \SL$ is proper, 
then we can choose $\wt F : M \setminus E \to \SL$ to be proper. 
\newcounter{MT}\setcounter{MT}{\value{enumi}}
\end{enumerate}
\end{theorem}
%
%
\begin{remark}\label{rk:E}
Under the natural assumption that $F:S\setminus E\to\SL$ does not extend as a holomorphic immersion to any point of $E$, we have that $E$ is the disjoint union $E = E_{\infty} \cup E_{\rm sing}$, where 
$E_{\infty}\subset E$ is the set of those points in $E$ where the meromorphic extension of $F$ to $\mathring S$ has an effective pole and $E_{\rm sing}= E\setminus E_\infty$ is the set where the extended $F$ is holomorphic and has a singular point (that is, those points of $E$ where the meromorphic extension of $f$ to $\mathring S$ has a zero). Therefore, $F$ is holomorphic on $M\setminus E_\infty$ with the singular locus $E_{\rm sing}$. In this situation, the holomorphic null curve $\wt F : M \setminus E \to \SL$ provided by Theorem \ref{th:mt} extends holomorphically to $M\setminus E_\infty$ with the singular locus $E_{\rm sing}=\{p\in M\setminus E_\infty\colon d\wt F_p=0\}$ and has an effective pole at every point in $E_\infty$.
Moreover, the holomorphic map $\wt F : M \setminus E_{\infty} \to \SL$ can always be chosen to be almost proper (cf.\ \ref{mtap}) and, if the given $F:S \setminus E_{\infty} \to \SL$ is proper, then $\wt F : M \setminus E_{\infty} \to \SL$ can be chosen to be a proper map (cf.\ \ref{mtp}).
\end{remark}
%
%
\begin{remark}\label{rk:intbs}
A minor modification of the proof of Theorem \ref{th:mt} allows to choose the interpolation set $\Lambda\subset M$ with $\Lambda \subset S$ and $\Lambda \cap (S \setminus \mathring{S}) \not = \varnothing$, whenever $k (p) =  0$ for every $p \in \Lambda  \setminus \mathring{S}$. This upgrade is well understood in the framework of minimal surfaces in $\r^n$ and holomorphic null curves in $\C^n$, $n\ge 3$; see e.g.\ \cite[Theorem 7.1 and Proof of Lemma 3.3]{AlarconVrhovnik2024X}. The details in our case are very similar and we omit them.
\end{remark}
%
%
\begin{proof}[Proof of Theorem \ref{th:mt} assuming Lemmas \ref{lemma:indf}, \ref{lemma:emb}, and \ref{lemma:proper}]
The assumptions guarantee that $S$ consists of at most countably many connected components.
We assume for simplicity of exposition that $S$ has infinitely many components; the proof is otherwise simpler. Likewise, we assume that $F|_{\Lambda}$ is injective and that $F : S \setminus E \to \SL$ is a proper map (in particular, $F$ has an effective pole at every point in $E$). In this special case, condition \ref{mtp} implies condition \ref{mtap}; we shall explain how to ensure \ref{mtap} in the general case at the end of the proof.

Fix $0 <\varepsilon_0 < \varepsilon/2$, let $M_0\subset M$ be a smoothly bounded simply connected compact domain with $S_0:= S \cap M_0= \varnothing$, and choose any holomorphic null embedding $F^0 : M_0 \to \SL$. 
Since $F^0$ has range in $\SL$, 
there is a number $\mu_0>0$ such that $\| F^0 (p)\|_{\infty} > \mu_0$ for all $p \in b M_0$ (we may choose, for instance, $\mu_0=1/2$).
Let 
\begin{equation}\label{eq:exh}
 M_1 \Subset M_2 \Subset \cdots \subset \bigcup_{n\in\n} M_n = M
\end{equation}
be an exhaustion of $M$ by connected smoothly bounded Runge compact domains such that $M_0 \subset \mathring M_1$ and for each $n\in\n$ we have that $S \cap b M_n = \varnothing$  and $\mathring M_n\setminus M_{n-1}$ contains a single connected component 
\begin{equation}\label{eq:sn}
	S_n := S \cap (M_n\setminus \mathring M_{n-1})
\end{equation}  
of $S$. Such an exhaustion exists by the assumptions on $S$.
It follows that $S \cap M_n = \bigcup_{j=1}^n S_j$ for all $n \in \N$ and $S = \bigcup_{n \in\N} S_n$. Also choose a sequence of smoothly bounded open simply connected domains $ \Delta_n \Subset M_n$ such that $E\cap M_n \subset \Delta_n$ and $ \Delta_n \cap \Lambda = \varnothing$ for each $n \in \N$. Moreover, make sure that $\overline{\Delta_{n+1}}\cap M_n \subset \Delta_n$ for every $n \in \N$ and
\begin{equation}\label{eq:Dn}
	\bigcup_{n\in\N} (M_n \setminus \Delta_n) = M \setminus E,
\end{equation}
equivalently, $\bigcap_{j\ge n}(\Delta_j\cap M_n)=E\cap M_n$ for all $n\in \N$.
Set
\begin{equation}\label{eq:defmun}
	\mu_n =\frac12 \min \{  \| F(p) \|_\infty : p \in S_n \setminus E\} >\frac14 >0,  \quad  n \in \N;
\end{equation} 
recall that $F$ has an effective pole at each point of $E$.
Since $F:S\setminus E\to \SL\subset\C^4$ is a proper map, we have that 
\begin{equation}\label{eq:mun}
	\lim_{n \to \infty} \mu_n = +\infty.
\end{equation}
To complete the proof, we shall inductively construct sequences of numbers $\varepsilon_n >0$  and nonflat holomorphic null embeddings $F^n : M_n \setminus E \to \SL$ extending meromorphically to $M_n$ with an effective pole at each point in $E\cap M_n$, and satisfying the following conditions for every $n\in \N$.
\begin{enumerate}[label=(\alph*$_n$)]
\item \label{seq-} $F^n - F^{n-1}$ is continuous (hence holomorphic) on $M_{n-1}$ and $F^n - F$ is continuous on $\bigcup_{j=0}^n S_j$ (hence of class $\Ascr^r(\bigcup_{j=0}^n S_j)^4$).
\item \label{seqa} $\|F^n - F^{n-1}\|_{r, M_{n-1}}< \varepsilon_{n-1}$ and $\|F^n - F\|_{r, S_n}< \varepsilon_{n-1}$. 
\item \label{seqi} $F^n - F$ vanishes to order $k(p)$ at every point $p \in (\Lambda \cup E ) \cap M_n$.
\item \label{seqep} $0<\varepsilon_{n}<\varepsilon_{n-1}/2$ and if $G \colon M_n \setminus E \to \SL$ is a holomorphic map such that $\| G-F^n \|_{r,M_n\setminus E}< 2 \varepsilon_n$, then $G$ is a nonflat embedding on $ M_n \setminus \Delta_n$.  
\newcounter{mtc}\setcounter{mtc}{\value{enumi}}
\item  \label{seqp1} $\|F^n (p)\|_{\infty} > \mu_n$ for all $p \in b M_n$.
\item  \label{sep2} $\| F^n (p) \|_{\infty} > \min \{ \mu_{n-1}, \mu_n \}$ for all $p \in M_n \setminus (\mathring M_{n-1} \cup E)$. 
\end{enumerate}
Suppose for a moment that such sequences exist. Then, by \eqref{eq:exh},  \ref{seq-}, \ref{seqa}, and \ref{seqep}, there is a limit holomorphic null curve
\[
\wt F:= \lim_{n\to \infty} F^n :M \setminus E \to \SL,
\]
such that $\wt F - F^n$ is continuous on $M_n$, $\wt F - F$ is continuous on $S = \bigcup_{n\ge 0} S_n$,
\begin{equation}\label{eq:eps}
\| \wt F - F^n \|_{r, M_n} < 2 \varepsilon_n<\varepsilon, \quad\text{and}\quad \| \wt F - F\|_{r, S_n} < 2 \varepsilon_{n-1} <  \varepsilon  \text{ for all } n \in \N.
\end{equation}
We claim that $\wt F$ satisfies the conclusions of the theorem.
Indeed, properties \ref{mt-} and \ref{mtapp} follow from \eqref{eq:eps} since $S=\bigcup_{n\ge1}S_n$. Likewise, \ref{mtint} follows from \eqref{eq:exh} and \ref{seqi}.
By \ref{seqep} and \eqref{eq:eps}, $\wt F$ is a nonflat embedding on $M_n \setminus \Delta_n$ for every $n \in \N$, hence $\wt F$ is injective on $M$ by \eqref{eq:Dn}, showing \ref{mtiny}.
Since $F : S \setminus E \to \SL$ is proper, condition  \ref{mtapp} ensures that $\wt F|_{S \setminus E} : S \setminus E \to \SL$ is proper as well. 
Therefore to check \ref{mtp}, it suffices to see that $\{\wt F (p_m)\}_{m\in \N}$ diverges in $\SL$ for every sequence $\{p_m \}_{m\in\N} \subset M \setminus S$ diverging on $M$. 
For this, \ref{sep2} and \eqref{eq:eps} imply that $| \wt F | > \min \{ \mu_{n-1}, \mu_n\} -\varepsilon$ everywhere on $M_n \setminus (\mathring M_{n-1} \cup E)$ for every $n \in \N$.
Thus, since  $\lim_{n \to \infty} \min\{\mu_{n-1}, \mu_n\} = +\infty$ by  \eqref{eq:mun}, condition \eqref{eq:exh} ensures that $\{\wt F (p_m)\}_{m \in \N}$ diverges for any sequence $\{p_m\}_{m \in \N}$ as above.

It remains to explain the induction. The base case is provided by the already chosen $\varepsilon_0$ and $F^0$, which satisfy (\hyperref[seqi]{\rm c$_{0}$}),  (\hyperref[seqp1]{\rm e$_{0}$}), and the second part of (\hyperref[seq-]{\rm a$_{0}$}),
while the other conditions are void.  For the inductive step, fix $n \in \N$ and suppose that we have a number $\varepsilon_{n-1} >0$ and a nonflat holomorphic null embedding $F^{n-1} : M_{n-1} \setminus E \to \SL$ extending meromorphically to $E\cap M_{n-1}$ with an effective pole at each point in $E\cap M_{n-1}$, and satisfying (\hyperref[seqi]{c$_{n-1}$}),  (\hyperref[seqp1]{e$_{n-1}$}), and the second part of (\hyperref[seq-]{\rm a$_{n-1}$}).
Let us then find $\varepsilon_n>0$ and a holomorphic null embedding $F^n : M_n \setminus E \to \SL$ extending meromorphically to $E\cap M_n$ with an effective pole at every point in $E\cap M_n$ and satisfying all conditions \ref{seq-}--\ref{sep2}. 

Set  $M_n ' = M_{n-1} \cup S_n\subset \mathring M_n$, which is a disjoint union; see \eqref{eq:sn} and recall that $S \cap (b M_{n-1} \cup b M_n) = \varnothing$. 
Note that $M_n '$ is a
Runge admissible compact subset and $(\Lambda \cup E)  \cap M_n ' = (\Lambda \cup E )  \cap M_n\subset \mathring M_n '$. Extend $F^{n-1}$ to $M_n ' \setminus E$ by setting 
\[
F^{n-1}(p)= F(p)\quad  \text{for all  }p\in S_n \setminus E.
\]
By (\hyperref[seqp1]{e$_{n-1}$}) and \eqref{eq:defmun}, we have that
\[
\| F^{n-1}\|_{\infty}>\min \{\mu_{n-1}, \mu_n \} \text{ on } b M_n'=b M_{n-1}\cup(S_n \setminus\mathring S_n)
\]
We may assume by Lemma \ref{lemma:indf} that $F^{n-1}$ is of class $\Oscr_\infty(M_n'|E, \SL)$, and then Lemma \ref{lemma:proper}
gives a nonflat holomorphic null curve $F^n : M_n \setminus E \to \SL$ satisfying \ref{seq-}, \ref{seqa}, \ref{seqi}, \ref{seqp1}, and \ref{sep2}; take into account (\hyperref[seqi]{c$_{n-1}$}), \eqref{eq:defmun}, and the second part of (\hyperref[seq-]{\rm a$_{n-1}$}).
By the second part of \ref{seq-}, $F^n$ extends meromorphically to $M_n$ with an effective pole at each point of $E \cap M_n$. 
Further, by 
\ref{seqi} and the initial assumption that $F|_\Lambda$ is injective, we have that so is $F^n|_{\Lambda\cap M_n}$, and hence
 Lemma \ref{lemma:emb} enables us to assume that $F^n$ is an embedding. The inductive step is then completed by choosing $\varepsilon_n>0$ so small that  \ref{seqep} is satisfied, which exists by Cauchy estimates. This concludes the proof in the case that $F : S \setminus E \to \SL$ is a proper map.
  
In the case when $F : S \setminus E \to \SL$ is not proper but has an effective pole at every point in $E$, then we follow the same argument but ignoring the numbers $\mu_n$ and condition \ref{sep2}, and replacing \ref{seqp1} by  
\begin{enumerate}[label= \rm (\alph*$_n '$)]
\setcounter{enumi}{\value{mtc}}
\item \label{seqap} $\| F^n (p)\|_{\infty} > n$ for all $p \in b M_n$.
\end{enumerate}
In this case, the limit map $\wt F \colon M \setminus E \to \SL$ satisfies
that $| \wt F | > n-\varepsilon$ on $b M_n$ for every $n \in \n$, by \eqref{eq:eps} and \ref{seqap}. Hence, taking into account \eqref{eq:exh} and that $\wt F$ has an effective pole at every point in $E$, it is an almost proper map, showing \ref{mtap}.

Finally, if $F : S \setminus E \to \SL$ is neither proper nor has an effective pole at every point in $E$, the same argument but ignoring \ref{seqp1} and \ref{sep2} gives a nonflat holomorphic null curve $\wt F : M \setminus E \to \SL$ satisfying the conclusion of the theorem.
\end{proof}

We finish this section by recording a Carleman type theorem (see \cite{Carleman1927}) for holomorphic null curves in $\SL$; see Corollary \ref{cr:Carleman}. 
Let $S$ be a closed set in an open Riemann surface $M$. The \emph{holomorphic hull of $S$} is defined as the union $\wh S = \bigcup_{n \in \n} \wh S_n$, where $\{S_n\}_{n \in \n}$ is any exhaustion of $S$ by compact sets and
$\wh S_n$ denotes the holomorphic convex hull of $S_n$, that is,  the union of $S_n$ and its holes.
The set $S \subset M$ has \emph{bounded exhaustion hulls} if $\wh{K \cup S} \setminus K \cup S$ is relatively compact in $M$ for every compact set 
$K \subset M$.
The set $S \subset M$ is \emph{Carleman admissible} 
if $\wh S = S$, $S$ has bounded exhaustion hulls, and $S = K \cup \Gamma$ where $K$ is the union of a locally finite pairwise disjoint collection of smoothly bounded compact domains and $\Gamma = \overline{S \setminus K}$ is the union of a locally finite pairwise disjoint collection of smooth Jordan arcs, so that each component of $\Gamma$ intersects $bK$ only at its endpoints (if at all) and all such intersections are transverse (see \cite[Def.\ 2.1]{Castro-InfantesChenoweth2020},  \cite[\textsection 3]{FFW18}, or \cite[\textsection3.8]{AlarconForstnericLopez2021Book}). 
Note that a compact set $S \subset M$ is Carleman admissible if and only if it is a Runge admissible set in the sense of Definition \ref{def:admissible}. 
The notion of \emph{generalized null curve $S\to\SL$ of class} $\Ascr_{\infty}^r (S|E)$ $(r \in \N)$ on a Carleman admissible set $S$ extends naturally that of generalized null curve of class $\Ascr_{\infty}^r (S|E)$ on admissible sets in Definition \ref{def:gmncsl}, just replacing admissible by Carleman admissible.
%
%
\begin{corollary}\label{cr:Carleman} 
Let  $M$ be an open Riemann surface, $\theta$ a nowhere vanishing holomorphic $1$-form on $M$, $S = K \cup \Gamma$ a Carleman admissible subset, and $\Lambda$ and $E$ disjoint closed discrete subsets of $M$ with $\Lambda \cup E \subset  \mathring S$.
Also let $(F ,f \theta)$ be a generalized null curve $S \setminus E \to \SL$ of class $\Ascr^1_{\infty} (S|E)$. 
Then,  for any continuous function $\varepsilon \colon S \to (0, +\infty)$ and any map $k : \Lambda\cup E  \to \n$ there is a nonflat holomorphic null curve  $\wt F : M \setminus E \to \SL$ of class $\Oscr_\infty(M|E, \SL)$ such that:
\begin{itemize}
\item  $\wt F - F$ is continuous on $S$.
\item  $| \wt F (p)- F(p)| < \varepsilon(p)$ for all $p \in S$. 
\item $\wt F - F$ vanishes to order $k (p)$ at every point $p \in \Lambda \cup E $.
\item If $F|_{\Lambda}$ is injective, then we can choose $\wt F : M\setminus E  \to \SL$ to be injective.		
\item If $F$ has an effective pole at every point in $E$, then we can choose $\wt F$ to be complete. 
\item If $F : S\setminus E \to \SL$ is proper, then we can choose $\wt F : M \setminus E  \to \SL$ to be proper. 
\end{itemize}
\end{corollary}

Note that the approximation in Corollary \ref{cr:Carleman} takes place in the fine Whitney topology on $S$.
The proof can be carried out from Theorem \ref{th:mt} by the same inductive argument as in the proof of \cite[Theorems 3.8.6 and 3.9.4]{AlarconForstnericLopez2021Book}, see also those of \cite[Theorem 5.2]{Castro-InfantesChenoweth2020} and \cite[Theorem 2.4]{Svetina2024}, and we omit the details.

%
%
\section{A step towards the proof of Lemma \ref{lemma:indf}}\label{sec:lindf}

\noindent A {\em holomorphic null curve} in $\C^3$ is a holomorphic immersion $M \to \C^3$, where $M$ is an open Riemann surface, whose derivative with respect to any local holomorphic coordinate on $M$ takes its values in the {\em punctured null quadric}
\begin{equation}\label{eq:nullquadric}
	\mathbf{A}_*  = \{ (z_1, z_2, z_3) \in \C^3 \colon z_1^2 + z_2^2 + z_3^2  = 0\} \setminus \{0\}.
\end{equation}
We will say that a holomorphic null curve which is also an embedding is a {\em holomorphic null embedding}.
Given a subset $W \subset M$,
we will say that a map $W \to \C^3$ is a {\it holomorphic null curve} if it extends to a holomorphic null curve on some open neighborhood of $W$ in $M$.
We also have the following analogue of Definition \ref{def:gmncsl}; cf. \cite[Definition 2.6]{AlarconLopez2019} and \cite[Remark 2.4]{AlarconVrhovnik2024X}.
%
%
\begin{definition}\label{def:gmnc}
Let $S = K \cup \Gamma$ be an admissible set (Definition \ref{def:admissible}) in an open Riemann surface $M$, $\theta$ a nowhere vanishing holomorphic $1$-form on a neighborhood of $S$, and  $E\subset \mathring{K}=\mathring{S}$ a finite subset.  A \emph{generalized null curve $S\setminus E \to \C^3$ of class $\Ascr^r_{\infty} (S| E)$ $(r \in \n)$} is a pair $(X,f \theta$) with $X \in \Ascr^r_{\infty} (S | E,  \C^3)$  and $f \in \Ascr^{r-1}_{\infty} ( S|E, \mathbf{A}_*)$ (see \eqref{eq:nullquadric}) satisfying the following.
\begin{itemize}
\item $f \theta =  d X$ holds on $K \setminus E$ (so $X : \mathring{K} \setminus E \to \C^3$ is a holomorphic null curve). 
\item For a smooth path $\alpha$ in $M$ parameterizing a connected component of $\Gamma$ we
have $ \alpha^* (f \theta) = \alpha^* (d X)= d ( X \circ \alpha)$.
\item $f^{-1} (\infty)  \subset E$ (see \eqref{eq:f-1}).
\end{itemize}
\end{definition}

As pointed out by Mart\'in, Umehara, and Yamada in \cite[Sec.\ 3.1]{MartinUmeharaYamada2009CVPDE}, the biholomorphism $\Tcal : \C^2 \times \C^* \to \SL \setminus \{ z_{11} = 0\}$ given in \eqref{eq:t} carries holomorphic null curves into holomorphic null curves. So does its inverse map $\Tcal^{-1}: \SL \setminus \{ z_{11} = 0\}\to \C^2 \times \C^*$,
\begin{equation}\label{eq:t-1}
\Tcal^{-1}  \left( \begin{array}{cc}
z_{11} & z_{12}\\
z_{21}& z_{22}
\end{array} \right) = \frac1{2z_{11}} \big( z_{21}+z_{12} \,,\, \igot(z_{21}-z_{12}) \, ,\, 2 \big).
\end{equation}
Since $\Tcal$ is a rational map it also takes meromorphic null curves into meromorphic null curves, but the polar sets need not be preserved. 
More precisely, given an admissible set $S$ in an open Riemann surface, a finite set $E \subset \mathring S$, and a generalized null curve $(X =(X_1, X_2, X_3), h \theta)$ from $S \setminus E$ into  $\C^3$ of class $\Ascr_{\infty}^r (S|E)$, with $X_3^{-1} (0) \subset \mathring S$, it turns out that 
\[
(\Tcal \circ X,  d \Tcal_X  (h \theta)) 
\]
is a generalized null curve $S \setminus (E \cup X_3^{-1} (0)) \to \SL$ of class $\Ascr_{\infty}^r (S|E \cup X_3^{-1} (0))$. 
Note that every point in $X_3^{-1} (0)$ is a pole of $\Tcal \circ X$, but there might be poles of $X$ in $E$ which are not poles of $\Tcal \circ X$.  
Conversely, if
$(F = (F_1, F_2, F_3, F_4), f \theta)$ is a generalized null curve $S \setminus E \to \SL$ of class $\Ascr_{\infty}^r (S| E)$, with $F_1^{-1} (0) \subset \mathring S$, then 
\[
(\Tcal^{-1} \circ F,  d \Tcal^{-1}_F ( f \theta))
\]
is a generalized null curve $S \setminus (E \cup F_1^{-1} (0)) \to \C^3$ of class $\Ascr_{\infty}^r (S|E\cup F_1^{-1} (0))$. 

The following result is the main technical tool in the proof of Lemma \ref{lemma:indf}. 
%
%
\begin{lemma} \label{lemma:ind}
Let $M$ be an open Riemann surface, $\theta$ a nowhere vanishing holomorphic $1$-form on $M$, $S \subset M$ a Runge admissible subset, 
$E$ and $\Lambda$ disjoint finite sets in $\mathring S$,
 $L \subset M$ a smoothly bounded compact domain with $S\subset \mathring L$, and $(X = (X_1,X_2,X_3), f \theta)$ a generalized null curve $S \setminus E \to \C^3$ of class $\Ascr^r_{\infty}  (S | E)$ $(r\in \n)$ such that $X_3^{-1} (0) \subset \mathring S$.
Given $\varepsilon>0$ and $k\in\n$, there is a nonflat holomorphic null curve $\wt X=(\wt X_1,\wt X_2,\wt X_3) :  L\setminus E \to \C^3$ of class $\Oscr_\infty(L|E, \C^3)$ satisfying the following.
\begin{enumerate}[label= \rm (\alph*)]
\item \label{conda} $\wt X - X$ is continuous on $S$, hence of class $\Ascr^r(S)^3$.
\item \label{conda1}  $\| \wt{X} - X \|_{r,S} < \varepsilon$. 
\item \label{condb} $\wt{X} - X$ vanishes to order $k$ at every point of $\Lambda \cup E$. 
\item \label{condc} $\wt X_3^{-1} (0) = X_3^{-1} (0)$.
\end{enumerate}
\end{lemma}
Note that the given curve $X$ need not have effective poles at the points in $E$, that is, it could extend holomorphically to some of them (see Remark \ref{rk:E}). Lemma \ref{lemma:ind} generalizes \cite[Lemma 2.3]{AlarconCastro-InfantesHidalgo2023}, which deals with the case $E = \varnothing$.
It also generalizes \cite[Theorem 3.1]{AlarconLopez2019} by adding the control on the zeros of the third component function. We defer the proof of  Lemma \ref{lemma:ind} to Section \ref{sec:c3}.

%
%
\begin{proof}[Proof of Lemma \ref{lemma:indf} assuming Lemma \ref{lemma:ind}]
Let us write $F = (F_1, F_2, F_3, F_4)$ and $f = (f_1, f_2, f_3, f_4)$.
We begin with the following. 

%
%
\begin{claim}\label{cl:sc}
Lemma \ref{lemma:indf} holds true under the extra assumption that $F_1^{-1}(0) \subset \mathring K$. Furthermore, in this case there is a holomorphic null curve $\wt F = (\wt F_1, \wt F_2, \wt F_3, \wt F_4)  : L \setminus E \to \SL$ satisfying the conclusion of the lemma such that $\wt F_1^{-1}(0) = F_1^{-1} (0)$. 
\end{claim}
\begin{proof}
Assume that $F_1^{-1}(0) \subset \mathring K$. Set 
\[
E_0 := E\cup F_1^{-1} (0) \subset \mathring K = \mathring S.
\] 
Note that $E_0$ is finite and $E_0 \setminus E = F_1^{-1} (0)$.
By the discussion preceding Lemma \ref{lemma:ind}, the pair
$(X := \Tcal^{-1} \circ F,  \,  d \Tcal^{-1}_F (f \theta))$
defines a generalized null curve $S \setminus E_0 \to \C^3$ of class $\Ascr_{\infty}^r (S | E_0)$; see \eqref{eq:t-1}. Write $X = (X_1, X_2, X_3)$ and note that $X_3 = 1/F_1$ vanishes nowhere on $S \setminus E_0$. 
Fix $\varepsilon_0 >0$ and $k _0 \in \n$ to be specified later. Lemma \ref{lemma:ind}  gives a nonflat holomorphic null curve $\wt X =(\wt X_1, \wt X_2, \wt X_3) :L \setminus E_0 \to \C^3$ such that:
\begin{enumerate}[label = \rm (\Roman*)]
\item \label{condalemma}$\wt X - X$ is continuous on $S$, hence of class $\Ascr^r(S)^3$.
\item \label{conda1lemma} $\| \wt X -X\|_{r, S} < \varepsilon_0$.
\item \label{condblemma} $\wt X - X$ vanishes to order $k_0$ at every point of $\Lambda \cup E_0$. 
\item \label{condclemma} $\wt X_3^{-1} (0) = X_3^{-1} (0) = \varnothing$.
\newcounter{l23}\setcounter{l23}{\value{enumi}}
\end{enumerate}
We claim that
$\wt F:=\Tcal \circ \wt X \colon L  \setminus E_0 \to \SL \setminus \{ z_{11} = 0\}$
(see \eqref{eq:t}) extends to a holomorphic null curve
\[
\wt F =  (\wt F_1, \wt F_2, \wt F_3, \wt F_4) :  L \setminus E \to \SL
\]
that satisfies the conclusion of the lemma.
Indeed, \ref{condclemma} shows that $\wt X$ has range in $\C^2 \times \C^*$, hence $\wt F$ is a holomorphic null curve on $L \setminus E_0$. In view of \ref{condalemma} and \ref{condblemma}, if $k_0$ is sufficiently large then $\wt F - F$ extends continuously to $S$ and vanishes to order $k$ on $\Lambda \cup E_0$. Since $F$ is continuous and an immersion on $E_0 \setminus E$, this shows that $\wt F$ is a holomorphic null curve on $L \setminus E$ as claimed; recall that $k\ge 1$. Taking also into account \ref{conda1lemma}, this discussion shows conditions \ref{f-}, \ref{fap}, and \ref{fint} in Lemma \ref{lemma:indf} whenever $\varepsilon_0>0$ is sufficiently small. 
Nonflatness of $\wt F$ follows from that of $\wt X$. 

For the second assertion in the statement of the claim, recall that $\wt F_1=1/\wt X_3$ on $L\setminus E_0$ by \eqref{eq:t}. Since $\wt X_3$ has poles only in $E_0$, $\wt F_1$ vanishes nowhere outside $E_0$, so $\wt F^{-1}_1 (0)= \wt F^{-1}_1(0) \cap (E_0\setminus E)$. Since \ref{condblemma} implies that $\wt F=F$ on $E_0\setminus E=F_1^{-1}(0)$,
we infer that $\wt F_1^{-1}(0) = F_1^{-1} (0)$.
\end{proof}

Since $F$ has the range in $\SL$, the map $(F_1, F_3) \colon S = K \cup \Gamma \to \C^2$ vanishes nowhere. Then the differential of the function
\begin{equation}\label{eq:trans}
\begin{split}
(bK \cup \Gamma) \times \C^2 & \to \C \\
(p,(a,b))& \mapsto  a F_1 (p) + b F_3 (p).\\
\end{split}
\end{equation}
with respect to $(a,b)$ has maximal rank for every $(p, (a, b)) \in (b K \cup \Gamma)\times \C^2$, hence Abraham's transversality theorem \cite{Abraham1963} (see also \cite[Theorem 1.4.3]{AlarconForstnericLopez2021Book}) shows that this map is transverse to $\{0\} \subset \C$ for almost every pair $(a,b) \in \C^2$. Since $b K$ and $\Gamma$ are smooth manifolds of real dimension $1$, there exist constants $a, b \in \C^*$ such that $a F_1 + b F_3 \not = 0$ on $b K \cup \Gamma$, so $F_1 + \lambda F_3 \not = 0$ on $b K \cup \Gamma$, where $\lambda = b/a \in \C^*$. 
Thus, the generalized null curve $(F^{\lambda}=(F^{\lambda}_1,F^{\lambda}_2,F^{\lambda}_3,F^{\lambda}_4), f^{\lambda}\theta)$ defined as
\[
F^{\lambda}=	\left(\begin{matrix}
	F_1 + \lambda F_3 & F_2 + \lambda F_4 \\
	F_3  & F_4
\end{matrix}\right) \quad\text{and}\quad
f^\lambda=\left(\begin{matrix}
	f_1 + \lambda f_2 & f_2 + \lambda f_4\\
	f_3 & f_4
\end{matrix}\right)
\]
satisfies $(F_1^{\lambda})^{-1}(0)\subset\mathring K$; see Remark \ref{rk:flambda}. 
By Claim \ref{cl:sc}, for any $\varepsilon^\lambda>0$ there is a nonflat holomorphic null curve $\wt F^{\lambda} = (\wt F^{\lambda}_1, \wt F^{\lambda}_2, \wt F^{\lambda}_3, \wt F^{\lambda}_4) : L \setminus E \to \SL$ such that $\wt F^\lambda  - F^\lambda$ is continuous on $S$, $\| \wt F^\lambda  - F^\lambda \|_{r, S} < \varepsilon^\lambda$, and $\wt F^\lambda - F^\lambda$ vanishes to order $k$ at every point of $\Lambda \cup E$. 
If $\varepsilon^{\lambda} >0$ is sufficiently small, then it is easily seen that the holomorphic null curve
\[
\wt F = 
\left(\begin{matrix}
	\wt F^{\lambda}_1  - \lambda \wt F^{\lambda}_3 & \wt F^{\lambda}_2 - \lambda \wt F^{\lambda}_4 \\
	\wt F^{\lambda}_3  & \wt F^{\lambda}_4
\end{matrix}\right) : L \setminus E \to \SL
\]
satisfies the conclusion of Lemma \ref{lemma:indf}.
\end{proof}

%
%
\section{Proof of Lemma \ref{lemma:emb}}\label{sec:lemma:emb}

\noindent
Given a map $h : A \to B$ between sets, we denote by
\begin{equation}\label{eq:Dh}
D_h = \{ p \in A \colon h^{-1} (h(p)) \setminus \{ p\} \not =\varnothing\}
\end{equation}
the set of double points of $h$.
It is clear that $h$ is injective if and only if $D_h = \varnothing$. 

Assume that $F|_\Lambda$ is injective and write $F= (F_1, F_2, F_3, F_4) : K \setminus E \to \SL$. Let us first prove the following approach to the lemma. 
%
%
\begin{claim}\label{cl:l23}
In the assumptions of Lemma \ref{lemma:emb}, if $F_1^{-1} (0) \subset \mathring K$ then there is a nonflat holomorphic null curve $\wt F:K\setminus E\to \SL$ satisfying $D_{\wt F} =  F_1^{-1} (0) \cap D_F$ (see \eqref{eq:Dh}) and conditions  \ref{f-em}, \ref{fapem}, and \ref{fintem} in the statement of the lemma.
\end{claim}
%
%
\begin{proof}
Assume that $F_1^{-1} (0) \subset \mathring K$, set $E_0 := E \cup F_1^{-1} (0) \subset \mathring K$, and consider the holomorphic null curve $X = (X_1, X_2, X_3) :=  \Tcal^{-1}\circ F : K \setminus E_0 \to \C^3$ extending meromorphically to $K$ (see \eqref{eq:t-1}). 
As in the proof of Claim \ref{cl:sc}, the set $E_0$ is finite,
\begin{equation}\label{eq:EE0}
	E_0\setminus E=F_1^{-1}(0),
\end{equation} 
and $X_3 = 1/F_1$ vanishes nowhere on $K \setminus E_0$.
Fix $\varepsilon_0>0$ and $k_0 \in \N$ to be specified later. By \cite[Theorem 4.1]{AlarconVrhovnik2024X}\footnote{The assumptions in \cite[Theorem 4.1]{AlarconVrhovnik2024X} require $X$ to have effective poles at all points in $E_0$. However, a very minor modification of the proof shows that this condition is not necessary. Indeed, it suffices to choose the function $g$ in \cite[Eq.\ (3.1)]{AlarconVrhovnik2024X} to be meromorphic on $M$ (instead of holomorphic) and such that $g$ is holomorphic and nowhere vanishing on $M\setminus E$ and the product $\tilde{f_0} = f g$ in \cite[Eq.\ (3.2)]{AlarconVrhovnik2024X} is continuous and nowhere vanishing on $S$. (Cf.\ \eqref{eq:f0} in the present paper.) With this modification and using Hurwitz theorem, the same proof leads to get rid of the mentioned assumption in the statement of \cite[Theorem 4.1]{AlarconVrhovnik2024X}. Cf.\ Proof of Claim \ref{cl:lemma} in the present paper.} there is a nonflat holomorphic null embedding $\wt X = (\wt X_1, \wt X_2, \wt X_3) : K \setminus E_0 \to \C^3$ such that: 
\begin{enumerate}[label = \rm (\Roman*)]
\item \label{alemma}$\wt X - X$ is continuous on $K$, hence holomorphic.
\item \label{blemma} $\| \wt X -X\|_{r, K} < \varepsilon_0$.
\item \label{clemma} $\wt X - X$ vanishes to order $k_0$ at every point of $\Lambda \cup E_0$. 
\end{enumerate}
Further, assuming that $\varepsilon_0>0$ and $k_0\in\N$ are sufficiently small and large, respectively, we have that
\begin{equation}\label{0lemma} 
	\wt X_3^{-1} (0) = \varnothing. 
\end{equation}
Indeed, argue by contradiction and suppose that $\wt X_3^{-1} (0) \not = \varnothing$. 
Since $X_3$ vanishes nowhere on $K \setminus E_0$, conditions \ref{blemma} and \ref{clemma} and Hurwitz theorem (see \cite{Hurwitz1895} or \cite[VII.2.5, p.148]{Conway1973}) imply that $\wt X_3 \equiv 0$ on $K \setminus E_0$ whenever $\varepsilon_0>0$ and $k_0\in\N$ are suitable, which contradicts the nonflatness of $\wt X$ and shows \eqref{0lemma}. Then, arguing as in the proof of Claim \ref{cl:sc}, we infer that the injective map $\wt F  = \Tcal \circ \wt X  : K  \setminus E_0 \to \SL \setminus \{ z_{11} = 0\}$ (see \eqref{eq:t}) is well defined and extends to a holomorphic null curve
\[
	\wt F  =  (\wt F_1, \wt F_2, \wt F_3, \wt F_4) : K  \setminus E \to \SL
\] 
satisfying conditions \ref{f-em}, \ref{fapem}, and \ref{fintem} in the statement of  Lemma \ref{lemma:emb}.
Moreover, by \ref{alemma}, \ref{clemma}, and \eqref{eq:EE0}, we have that
\[
	\wt F_1^{-1}(0)=F_1^{-1}(0)=E_0\setminus E.
\]
This and the injectivity of $\wt F$ on $K \setminus E_0$ show that $D_{\wt F}\subset  F_1^{-1}(0)$. Finally, as $\wt F=F$ on $F_1^{-1} (0)$ (see \eqref{eq:EE0} and \ref{condblemma}) we obtain that $D_{\wt F}= F_1^{-1}(0)\cap D_F$.
\end{proof}      

Since $F_1$ is not constant by nonflatness of $F$, we may assume, up to passing to a slightly larger $K$ if necessary, that $F_1^{-1} (0) \subset \mathring K$. Choose numbers $\varepsilon>0$, $k\in\N$, and $r\in\Z_+$, and fix $\varepsilon_0>0$ to be specified later. We shall now complete the proof of the lemma by two succesive applications of Claim \ref{cl:l23}. 

\noindent
{\em First deformation.} By Claim \ref{cl:l23}, there is a nonflat holomorphic null curve $\wh F : K \setminus E \to \SL$ satisfying the following conditions.
\begin{enumerate}[label = \rm (\alph*)]
\item \label{24a} $\wh F  - F$ is continuous on $K$.
\item \label{24b}  $\| \wh F  - F \|_{r, K} < \varepsilon_0$.
\item \label{24c} $\wh F - F$ vanishes to order $k$ at every point of $\Lambda \cup E$. 
\item \label{24d} $D_{\wh F} =  F_1^{-1} (0) \cap D_F$; see \eqref{eq:Dh}.
\end{enumerate}

Condition \ref{24d} implies that the restriction of $\wh F$ to $(K\setminus E)\setminus (F_1^{-1} (0) \cap D_F)$ is injective, so its double points lie in $F_1^{-1} (0)$ and we still have to deal with them. 

\noindent
{\em Second deformation.} Since $F_1^{-1}(0)\subset\mathring K$ is finite,  $D_{\wh F}$ is finite as well in view of \ref{24d}. Thus, arguing as in the proof of Lemma \ref{lemma:indf} but replacing $b K \cup \Gamma$ by $b K \cup D_{\wh F}$ (see \eqref{eq:trans}), the transversality argument furnishes us with a constant $\lambda \in \C^*$ such that the nonflat holomorphic null curve
\begin{equation}\label{eq:defFhl}
\wh F^{\lambda} :K  \setminus E \to \SL, \quad \wh F^\lambda
= \left(\begin{matrix}
		\wh F_1^{\lambda}  & \wh F_2^{\lambda}  \\
		\wh F_3^{\lambda}  & \wh F_4^{\lambda}
\end{matrix}\right) 
= \left(\begin{matrix}
		\wh F_1  + \lambda \wh F_3 & \wh F_2 + \lambda \wh F_4 \\
		\wh F_3  & \wh F_4
\end{matrix}\right) 
\end{equation}
(see Remark \ref{rk:flambda}) satisfies
\begin{equation}\label{eq:nozeros}
	(\wh F^{\lambda}_1)^{-1}(0) \cap D_{\wh F} = \varnothing
	\quad\text{and}\quad (\wh F^{\lambda}_1)^{-1}(0) \subset \mathring K.
\end{equation}
Since $F|_\Lambda$ is injective, so is $\wh F|_\Lambda$ by \ref{24c}, and hence so is  $\wh F^{\lambda}|_{\Lambda}$ (see Remark \ref{rk:flambda}). This and the second condition in \eqref{eq:nozeros} enable us to apply Claim \ref{cl:l23} once again to find a nonflat holomorphic null curve
$\wt F^{\lambda} = (\wt F^{\lambda}_1, \wt F^{\lambda}_2, \wt F^{\lambda}_3, \wt F^{\lambda}_4) :K \setminus E \to \SL$ such that:
\begin{enumerate}[label = \rm (\alph*)]
	\setcounter{enumi}{\value{l23}}
	\item \label{h-}$  \wt F^{\lambda}- \wh F^{\lambda}$ is continuous on $K$.
	\item \label{hap}  $\|\wt F^{\lambda}    - \wh F^{\lambda} \|_{r, K} < \varepsilon_0$.
	\item \label{hint} $\wt F^{\lambda} - \wh F^{\lambda}$ vanishes to order $k$ at every point of $\Lambda \cup E$. 
	\item \label{hembd} $D_{\wt F^{\lambda}} = (\wh F^{\lambda}_1)^{-1}(0) \cap D_{\wh F^{\lambda}}$.
\end{enumerate}
Since $D_{\wh F^{\lambda}} = D_{\wh F}$ (see \eqref{eq:Dh} and Remark \ref{rk:flambda}), the first part of \eqref{eq:nozeros} and \ref{hembd} ensure that $D_{\wt F^{\lambda}}=\varnothing$, that is, $\wt F^\lambda:K\setminus E\to\SL$ is an embedding. Using conditions \ref{24a}--\ref{24c} and \ref{h-}--\ref{hint}, as well as the definition of $\wh F^\lambda$ in \eqref{eq:defFhl}, it is now easily seen that the nonflat holomorphic null curve
\[
	\wt F = 
	\left(\begin{matrix}
		\wt F^{\lambda}_1  - \lambda \wt F^{\lambda}_3 & \wt F^{\lambda}_2 - \lambda \wt F^{\lambda}_4 \\
		\wt F^{\lambda}_3  & \wt F^{\lambda}_4
	\end{matrix}\right) : K \setminus E \to \SL
\]
satisfies the conclusion of Lemma \ref{lemma:emb} provided that $\varepsilon_0>0$ is sufficiently small. Note that $\wt F$ is injective, hence an embedding, since so is $\wt F^{\lambda}$ (see Remark \ref{rk:flambda}). 

%
%

\section{Proof of Lemma \ref{lemma:proper} assuming  Lemma \ref{lemma:indf}}\label{sec:lproper}
\noindent 
The proof of Lemma \ref{lemma:proper} follows a standard argument developed by Alarc\'on and L\'opez in \cite{AlarconLopez2012JDG} for constructing proper minimal surfaces in Euclidean space $\r^3$; see, in particular, \cite[Lemma 5.1]{AlarconLopez2012JDG}. Their method has been adapted to obtain proper surfaces with arbitrary complex structure in several frameworks, such as minimal surfaces in $\r^n$ and holomorphic null curves in $\c^n$ with $n\ge 3$ (see \cite[Lemma 3.11.1]{AlarconForstnericLopez2021Book}), directed holomorphic and meromorphic immersions of open Riemann surfaces into $\C^n$ (see \cite[Lemma 8.2]{AlarconForstneric2014IM} and \cite[Theorem 1.1]{AlarconVrhovnik2024X}), holomorphic Legendrian curves in $\C^{2n+1}$ (see \cite[Lemma 5.2]{AlarconForstnericLopez2017CM} or \cite[Lemma 3.7]{Svetina2024}), and holomorphic Legendrian curves in $\SL$ (see \cite[Lemma 4.3]{Alarcon2017JGA}). 
In order to adapt this construction method to our current aim,
we need the following tools.
\begin{enumerate}[label= (\rm \Alph*)]
\item \label{A} A semiglobal Runge-Mergelyan approximation type theorem with jet interpolation for generalized meromorphic null curves on admissible sets into $\SL$ (see Definitions \ref{def:admissible} and \ref{def:gmncsl}). 
\item \label{B}A semiglobal Runge approximation type result with jet interpolation for meromorphic null curves in $\SL$ with a fixed component function. 
\item \label{C} Given an admissible set $S=K\cup\Gamma$ in an open Riemann surface and a finite set $E\subset\mathring K$, a way to extend a given meromorphic null curve $K\setminus E\to\SL$ of class $\Ascr^1_\infty(K|E, \SL)$ to a generalized meromorphic null curve $S\setminus E\to\SL$ of class $\Ascr^1_\infty(S|E)$ with suitable control on its values in $\Gamma$.
\item \label{D} Given a compact set $K\subset M$, a holomorphic null curve $F=(F_1,F_2,F_3,F_4):K\to\SL$, and an index $i\in\{1,2,3,4\}$, a way to define a holomorphic null curve $\wh F=(\wh F_1,\wh F_2,\wh F_3,\wh F_4):K\to\SL$ with $\wh F_i=F_i$ and $|\wh F|$ as large as desired everywhere on $K$. 
\end{enumerate}
Tool \ref{A} is already provided by Lemma \ref{lemma:indf}. The following Lemma \ref{lemma:fix}, Lemma \ref{lemma:ext}, and Claim \ref{cl1} grant tools \ref{B}, \ref{C}, and \ref{D}, respectively. 

%
%
\begin{lemma}\label{lemma:fix}
Let $M$ be an open Riemann surface,  $K$ and $L$ smoothly bounded Runge compact domains in $M$ with $K\subset  \mathring{L}$,
$\Lambda$ and $E$ disjoint finite sets in $\mathring K$, and $i \in  \{ 1,2,3,4\}$, and assume that $F= (F_1, F_2, F_3, F_4)\colon K \setminus E \to \SL$ is a nonflat holomorphic null curve of class $\Oscr_\infty(K|E, \SL)$
such that the component function
$F_i$ extends to $L \setminus E$ as a holomorphic function with
\begin{equation}\label{eq:cerosint}
	F_i^{-1} (0) \subset \mathring K.
\end{equation}
Then, given $\varepsilon>0$,  $k\in\n$, and $r \in \Z_+$, there is a nonflat holomorphic null curve $\wt F = (\wt F_1, \wt F_2, \wt F_3, \wt F_4) \colon L\setminus E \to \SL$ of class $\Oscr_\infty(L|E, \SL)$ such that:
\begin{enumerate}[label= \rm (\roman*)]
\item \label{fija1} $\wt F - F$ is continuous on $K$, hence holomorphic.
\item \label{fija2} $\| \wt{F} - F \|_{r,K} < \varepsilon$.
\item \label{fija3} $\wt{F} - F$ vanishes to order $k$ at every point of $\Lambda \cup E$. 
\item \label{fija4}$\wt F_i = F_i$ everywhere on $L\setminus E$. 
\end{enumerate}
\end{lemma}
%
%
\begin{proof} 
We assume without loss of generality that $i = 1$. 
Fix a nowhere vanishing holomorphic $1$-form $\theta$ on $M$, write $f_1 = d F_1 / \theta$, and note that $f_1 \not \equiv 0$ by nonflatness of $F$.
Up to slightly enlarging $K$ if necessary, we then may also assume that 
\begin{equation}\label{eq:f1bk}
	f_1^{-1} (0) \cap b K = \varnothing.
\end{equation}
Consider the finite set $E_0 := E \cup F_1^{-1}(0) \subset \mathring K$ and the holomorphic null curve $X = (X_1, X_2, X_3):= \Tcal^{-1} \circ F : K \setminus E_0 \to \C^3$; see \eqref{eq:cerosint} and \eqref{eq:t-1}. Since $X_3 = 1/ F_1$ on $K \setminus E_0$ and $F_1$ extends holomorphically to $L \setminus E$ by hypothesis, it follows that $X_3$ extends holomorphically to $L \setminus E_0$. 
Fix $\varepsilon_0>0$ (small) and $k_0 \in \N$ (large) to be specified later. 
Denote $h_3 = d X_3 / \theta$ and observe that
$h_3= -f_1 /F_1^2$ vanishes nowhere on $bK$
by \eqref{eq:f1bk}. Thus
\cite[Theorem 6.1]{AlarconVrhovnik2024X}\footnote{The assumptions in \cite[Theorem 6.1]{AlarconVrhovnik2024X} require $X$ to have effective poles at all points in $E_0$ but, as in the aforementioned case of \cite[Theorem 4.1]{AlarconVrhovnik2024X}, this condition is actually not necessary.}  gives  a holomorphic null curve $\wt X = (\wt X_1, \wt X_2, \wt X_3): L \setminus E_0 \to \C^3$  such that: 
\begin{enumerate}[label= \rm (\Roman*)]
	\item \label{fix-} $\wt X - X$ is continuous on $K$, hence holomorphic.
	\item \label{fixap} $\| \wt X -X\|_{r, K} < \varepsilon_0$.
	\item \label{fixint} $\wt X - X$ vanishes to order $k_0$ at every point of $\Lambda \cup E_0$.
	\item \label{fixx3} $\wt X_3 = X_3$ everywhere on $L\setminus E_0$.
\end{enumerate}
Assuming that $\varepsilon_0>0$ and $k_0\in\N$ are suitably chosen and arguing as in the proof of Lemma \ref{lemma:indf}, the map $\Tcal \circ \wt X : L \setminus E_0 \to \SL$ extends to a nonflat holomorphic null curve 
$\wt F = (\wt F_1, \wt F_2, \wt F_3, \wt F_4) : L \setminus E \to \SL$
that satisfies \ref{fija1}, \ref{fija2}, and \ref{fija3} in view of \ref{fix-}, \ref{fixap}, and \ref{fixint}.
It also follows from \ref{fixx3} that $\wt F_1 = F_1$ on $L \setminus E_0$; take into account that $\wt F_1 = 1 / \wt X_3= 1/ X_3=F_1$ on $L \setminus E_0$ by \eqref{eq:t} and \eqref{eq:t-1}. By continuity, this implies that $\wt F_1 = F_1$ on $L \setminus E$, proving \ref{fija4}.
\end{proof}

%
%
\begin{lemma}\label{lemma:ext}
Let  $M$ be an open Riemann surface,  $\theta$ a nowhere vanishing holomorphic $1$-form on $M$, and $\Gamma$ a Jordan arc in $M$ with endpoints $p, q \in M$. Then, given $A, B\in \SL$, $A', B'\in \mathfrak{A}_*$ (see \eqref{eq:nullquadricsl2}), and an open connected set $\Omega \subset \SL$ with $A, B\in \Omega$, there exists a nonflat generalized null curve $(F, f \theta)$ from $\Gamma$ to  $\SL$ of class $\Ascr^1 (\Gamma)$ ($= \Ascr_{\infty}^1(\Gamma |\varnothing)$; see Definition \ref{def:gmncsl}) 
satisfying $F(\Gamma) \subset \Omega$, $F(p) = A$, $F(q) = B$, $f(p) = A'$, and $f(q) = B'$.
\end{lemma}
%
%
\begin{proof}
We may assume without loss of generality that $A$, $B$, and $\Omega$ are contained in $\SL \setminus \{z_{11} = 0\}$.
By \cite[Lemma 3.3]{AlarconCastro-Infantes2019APDE} applied to $\Tcal^{-1} (A), \Tcal ^{-1}(B)$, the open connected set $\Tcal^{-1} (\Omega)$, $d \Tcal_A^{-1} (A')$, and $d \Tcal_B^{-1}(B')$ (see \eqref{eq:t-1}), there is a nonflat generalized null curve $(X, h\theta)$ from $\Gamma$ to $\C^3$ of class $\Ascr^1 (\Gamma) = \Ascr_{\infty}^1 (\Gamma|\varnothing)$, such that
$X (\Gamma) \subset \Tcal^{-1} (\Omega)\subset\C^2\times\C^*$, $X (p)= \Tcal^{-1} (A)$, $X(q) = \Tcal ^{-1}(B)$, $h(p) = d \Tcal_A^{-1} (A')$, and $h(q)=  d \Tcal_B^{-1}(B')$.
Since $X$ has range in $\C^2 \times \C^*$ 
and is of class $\Ascr^1(\Gamma)$, the pair
$
(F,f\theta)=(\Tcal \circ X, d \Tcal_X (h \theta))
$
is a generalized null curve $\Gamma \to \SL \setminus \{z_{11} = 0\}$ of class $\Ascr^1 (\Gamma)$ according to the discussion preceding Lemma \ref{lemma:ind}.
It is then clear that $(F,f\theta)$ satisfies the required conditions in view of 
\eqref{eq:t}.
\end{proof}

%
%
\begin{claim}\label{cl1}
Let $M$ be an open Riemann surface, $K \subset M$ a compact set, and $F = (F_1, F_2, F_3, F_4) :  K \to \SL$ a holomorphic null curve. Given $i \in \{ 1,2,3,4\}$ and $\delta >0$, there is a holomorphic null curve $\wt F  = (\wt F_1, \wt F_2, \wt F_3, \wt F_4): K \to \SL$ such that $\wt F_i(p) = F_i(p)$ and $\max \big\{|\wt F_j (p) |\colon j \in\{1,2,3,4\}\setminus\{i\}\big\} > \delta$ for every $p \in K$.
\end{claim}
%
%
\begin{proof}
We assume that $i=1$;  the proof is analogous otherwise. Since the map $(F_1, F_2)\colon K \to \C^2$ 
vanishes nowhere on the compact set $K$, there exists $\lambda \in \C^*$ with
\[
	\max \{ | F_3 +  \lambda F_1 | , \, |F_4 + \lambda F_2 | \} > \delta \quad \text{everywhere on } K. 
\]
Then, the holomorphic null curve
\[
	\wt F = 
	\left(\begin{matrix}
		\wt F_1  & \wt F_2 \\
		\wt F_3   & \wt F_4
	\end{matrix}\right) := 
	\left(\begin{matrix}
		F_1  & F_2 \\
		F_3 + \lambda F_1 & F_4+ \lambda F_2
	\end{matrix}\right) 
	: K \to \SL
\]
satisfies the conclusion of the claim; see Remark \ref{rk:flambda}.
\end{proof}
As pointed out, with these tools in hand the proof of Lemma \ref{lemma:proper} is a straightforward adaptation of those in the sources  mentioned at the beginning of this section, so we only provide the following.
%
%
\begin{proof}[Sketch of proof of Lemma \ref{lemma:proper} assuming  Lemma \ref{lemma:indf}]
We assume without loss of generality that $L$ is connected and, by a finite recursive application of Lemmas \ref{lemma:ext} and \ref{lemma:indf} depending on the topology of $L\setminus\mathring K$, that $K$ is a strong deformation retract of $L$; see, e.g., \cite[Proof of Theorem 3.10.3]{AlarconForstnericLopez2021Book}.  Also, for simplicity of exposition, we assume that $L \setminus \mathring K$ is connected, hence a smoothly bounded compact annulus. Finally, we assume that $1/\varepsilon > \delta$ and, by nonflatness of $F$ and up to slightly enlarging $K$ if necessary, that
\begin{equation}\label{eq:F1234}
	F_i^{-1}(0)\subset \mathring K \quad \text{for all $i\in\{1,2,3,4\}$.}
\end{equation} 
The construction consists of two steps. 

\noindent{\em First deformation}. By \eqref{eq:F1234} and using Lemmas \ref{lemma:ext} and \ref{lemma:indf}, we reason as in the aforementioned sources (see e.g.\ \cite[Proof of Lemma 3.11.1, Step 1]{AlarconForstnericLopez2021Book}) in order to find
\begin{itemize}
\item smooth Jordan arcs $\{ \alpha_j \colon j \in \Z_n\} \subset b K$ ($n\ge 3$)
such that $\alpha_j$ and $\alpha_{j+1}$ have a common endpoint $p_j$ and are otherwise disjoint for $j \in \Z_n= \{0,1, \ldots, n-1\}$,  
\item smooth Jordan arcs $\{\beta_j \colon j \in \Z_n \} \subset b L$ 
such that $\beta_j$ and $\beta_{j+1}$ have a common endpoint $q_j$ and are otherwise disjoint for every $j \in \Z_n$, 
\item pairwise disjoint smooth Jordan arcs $\{ \gamma_j \colon j \in \Z_n \} \subset L \setminus \mathring K$ such that each $\gamma_j$ has $p_j$ and $q_j$ as endpoints and is otherwise disjoint from $bK\cup bL$, and the Runge compact set $S=K\cup\bigcup_{j\in\Z_n}\gamma_j\subset M$ is admissible,
\item closed discs $\{ D_j \colon j \in \Z_n \}\subset L \setminus \mathring K$ determined by $b D_j = \gamma_{j-1}  \cup \alpha_j  \cup \gamma_j \cup \beta_j$, $j \in \Z_n$ (so $\bigcup_{j \in \Z_n}  D_j = L \setminus \mathring K$; see Figure \ref{fig:proper} or  \cite[Fig.\ 3.1]{AlarconForstnericLopez2021Book}),
\item smoothly bounded closed discs $\{\Omega_j \colon j \in \Z_n\}$ with $\Omega_j \subset D_j \setminus ( \gamma_{j-1} \cup \alpha_j \cup \gamma_j)$ and 
$\Omega_j \cap \beta_j$ being a (connected) smooth Jordan arc for every $j \in \Z_n$, 
\end{itemize}
\begin{figure}[h!]
\begin{center}
\includegraphics[scale=.9]{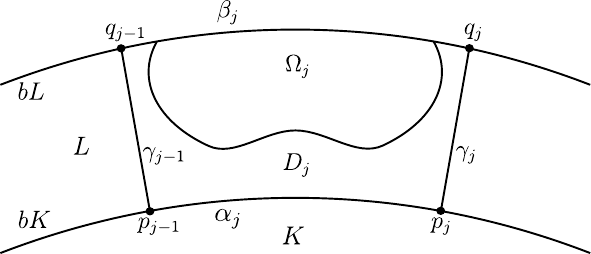}
\caption{Sets in the proof of Lemma \ref{lemma:proper}}
\label{fig:proper}
\end{center}
\end{figure}
and a nonflat holomorphic null curve $\wh F = (\wh F_1, \wh F_2, \wh F_3, \wh F_4) : L \setminus E \to \SL$ satisfying the following conditions.
\begin{enumerate}[label = \rm (\Roman*)]
\item $\wh F - F$ is continuous on $K$.  
\item $\| \wh F - F \|_{r, K} <\varepsilon/2$.
\item $\wh F- F$ vanishes to order $k$ at every point of $\Lambda \cup E$.
\item \label{ceros} $\wh F_i^{-1}(0)\subset \mathring K \cup \bigcup_{j \in \Z_n} \mathring \Omega_j$ for all $i\in\{1,2,3,4\}$.
\item \label{big} There is a map $\tau:\Z_n\to \{1,2,3,4\}$ such that $|\wh F_{\tau(j)} (p)| > \delta$ for all $p \in \overline{D_j \setminus \Omega_j}$  and $| \wh F_{\tau(j)} (p)| > 1 / \varepsilon$ for all $p \in \overline{\beta_j \setminus \Omega_j}$, $j\in\Z_n$.
\end{enumerate}
Condition \ref{ceros}, which is a technical addition to previous constructions, is possible by \eqref{eq:F1234}. Indeed, just define $\wh F$ on $\bigcup_{j\in\Z_n}\gamma_j$ in a suitable way and then choose the discs $\Omega_j$, $j\in\Z_n$, sufficiently big.

\noindent{\em Second deformation.} We keep arguing as in the aforementioned sources (see e.g.\ \cite[Proof of Lemma 3.11.1, Step 2]{AlarconForstnericLopez2021Book}) to obtain, via a  four steps  recursive application of Claim \ref{cl1} and Lemma \ref{lemma:fix} (taking into account \ref{ceros} when invoking the latter), a nonflat holomorphic null curve $\wt F =  (\wt F_1, \wt F_2, \wt F_3, \wt F_4) : L \setminus E \to \SL$ such that: 
\begin{itemize}
\item $\wt F - \wh F$ is continuous on $K$.
\item $\| \wt F - \wh F \|_{r, K}  <\varepsilon/2$.
\item $\wt F- \wh F$ vanishes to order $k$ at every point of $\Lambda \cup E$.  
\item $|\wt F_{\tau(j)} (p)| > \delta$  for all $p \in \overline{D_j \setminus \Omega_j}$  and $| \wt F_{\tau(j)} (p)| > 1 / \varepsilon$ for all $p \in \overline{\beta_j \setminus \Omega_j}$, $j\in\Z_n$,
where $\tau$ is the map in \ref{big}. 
\item $\max \big\{|\wt F_l (p) |\colon l \in\{1,2,3,4\}\setminus\{\tau(j)\}\big\} > 1 / \varepsilon$ for all $p\in\Omega_j$ and $j\in\Z_n$.
\end{itemize}
It is then clear from the above two lists of properties that the holomorphic null curve $\wt F$ satisfies the conclusion of Lemma \ref{lemma:proper}.
\end{proof}

%
%

\section{Completion of the proofs}\label{sec:c3}
\noindent
In this section we prove Lemma \ref{lemma:ind}. This shall complete the proof of Lemma \ref{lemma:indf}, and hence those of Lemma \ref{lemma:proper} and Theorem \ref{th:mt}.

If $S$ is set in an open Riemann surface, $M$, and $f \in \Oscr_{\infty} (S)$, $f \not \equiv 0$, we denote by
\begin{equation}\label{eq:orden}
	\ord_p (f) \in \Z
\end{equation}
the only integer such that $z^{-\ord_p (f)} f$ is holomorphic and nonzero at $p\in S$, where $z$ is any local holomorphic coordinate on $M$ with $z(p)=0$. Thus, $\ord_p (f)$ is the zero order of $f$ at $p$ when $\ord_p (f)>0$ and minus the pole order of $f$ at $p$ when $\ord_p (f)<0$, while $\ord_p (f) = 0$ means that $f$ is continuous at $p$ and $f(p) \not = 0$. 
%
%
\begin{proof}[Proof of Lemma \ref{lemma:ind}]
We assume that $ E \not = \varnothing$, otherwise the lemma holds true by \cite[Lemma 2.3]{AlarconCastro-InfantesHidalgo2023}. 
Up to passing to a larger $L$ if necessary, we may assume that $L \subset M$ is connected and Runge.
Write $f= (f_1, f_2, f_3)$. We may also assume without loss of generality that
\begin{equation}\label{eq:EX3Lambda}
		X_3^{-1} (0) \subset \Lambda
\end{equation}
and
\begin{equation}\label{eq:k>}
		k > 1 + \max\big\{ \sup\{  \ord_p(X_3) \colon p\in X_3^{-1}(0)\}\,,\, \max \{ o(p) \colon p \in E\}  \big\}\ge 1,
\end{equation}
where $o(p) = \max\{ |\ord_p (f_i)|  \colon i= 1,2,3\}$ for each $p \in E$; 
recall that $\varnothing\neq E \subset \mathring S$ and see \eqref{eq:orden}. 
%
%
Consider the following assertion.
\begin{claim}\label{cl:lemma}
The conclusion of the lemma holds if $S$ is a connected smoothly bounded compact domain which is a strong deformation retract of $L$ and $X : S \setminus E \to \C^3$ is a nonflat holomorphic null curve of class $\Oscr_\infty(S|E,\C^3)$.
\end{claim}

Let us reduce the proof of the lemma to the particular case in Claim \ref{cl:lemma}.
For this, \cite[Theorem 3.1]{AlarconLopez2019} (see also \cite[Theorem 3.8.2]{AlarconForstnericLopez2021Book} and \cite[Theorem 5.1]{AlarconVrhovnik2024X})\footnote{Once again, the assumptions in these results require $X$ to have effective poles at all points in $E$ but, as in the aforementioned cases of \cite[Theorem 4.1]{AlarconVrhovnik2024X} and \cite[Theorem 6.1]{AlarconVrhovnik2024X}, this condition is not necessary.} furnishes us with an open domain $U\supset L$ and a nonflat holomorphic null curve $Y=(Y_1,Y_2,Y_3) \colon U \setminus E \to \C^3$ such that $Y-X$ is continuous on $S$, $Y - X $ is as close as desired to $0\in \C^3$ in the $\Cscr^r$ norm on $S$, and $Y - X$ vanishes to order $k$ everywhere on $\Lambda \cup E$. Then, $Y|_{L\setminus E}$ formally satisfies conditions \ref{conda}, \ref{conda1},  and \ref{condb} in the statement of the lemma, but it need not meet \ref{condc}. Taking \ref{condb}, \eqref{eq:EX3Lambda}, and \eqref{eq:k>} into account, and assuming that $Y-X$ is close enough to $0$ on $S$, Hurwitz theorem (see \cite{Hurwitz1895} or \cite[VII.2.5, p.148]{Conway1973}) guarantees that $Y_3^{-1}(0)\cap  S=X_3^{-1}(0)\subset\Lambda\subset\mathring S$, while $Y_3$ has at most finitely many zeros on $\mathring L\setminus S$. We can then find a connected smoothly bounded Runge compact domain $S'\subset M$ such that  $S \Subset S' \Subset L$,  $S'$ is a strong deformation retract of $L$, and $Y_3$ vanishes nowhere on $S'\setminus S$. Replacing $X$ by $Y|_{S'\setminus E}$, which formally satisfies   \ref{condc}, reduces the proof of Lemma \ref{lemma:ind} to proving Claim \ref{cl:lemma}. Cf.\ \cite[Proof of Lemma 2.3]{AlarconCastro-InfantesHidalgo2023}.

%
%
\begin{proof}[Proof of Claim \ref{cl:lemma}]
The proof is a modification of that of \cite[Lemma 2.3]{AlarconCastro-InfantesHidalgo2023} that allows to deal with poles. In particular, we proceed with a typical method of controlling periods as in \cite[Ch.\ 3]{AlarconForstnericLopez2021Book} and, in addition, take ideas from  \cite{AlarconLopez2019}. We split the proof in three main steps.

\noindent {\em Step 1: The spray.} We shall construct a suitable period dominating spray of holomorphic maps with range in the punctured null quadric $\mathbf{A}_*$; see \eqref{eq:nullquadric} and \eqref{eq:sprayf0}.
For this, fix a point $p_0\in\mathring S\setminus \Lambda$ and a finite family $ \{ C_j \}_{j = 1}^n$ of smooth Jordan arcs and curves in $\mathring S$ satisfying the following properties (see \cite{FarkasKra1992} or \cite[Lemma 1.12.10]{AlarconForstnericLopez2021Book}):
\begin{itemize}
\item $C_i\cap C_j=\{p_0\}$ for any $i\neq j\in\{1,\ldots,n\}$.
\item $\{ C_1, \ldots, C_{\mu} \}$, where $\mu\in\{0,\ldots,n\}$ is the cardinal of $\Lambda$, are smooth Jordan arcs with the initial point $p_0$ and the final point in $\Lambda$.
\item  $\{ C_{\mu + 1}, \ldots, C_n\}$ are smooth oriented Jordan curves which together span the first homology group $\Hcal_1 ( S, \z)=\Hcal_1 ( L, \z)\cong \Z^{n-\mu}$.
\item The union $C = \bigcup_{j=1}^n C_j\subset \mathring S$ is a Runge compact set in $M$ and $\bigcup_{j=1}^n C_j$ is a strong deformation retract of $S$, hence of $L$.
\end{itemize}
Denote
\[
	\Ccal^0 (C, f) = \{ h \in \Ccal^0 ( C \setminus E, \mathbf{A}_* ) \colon h - f \in \Ccal^0 (C, \C^3) \},
\]
and let  $\Qcal= ( \Qcal_1,\ldots, \Qcal_n) :  \Ccal^0 (C, f) \to (\C^3)^n$ be the period map defined by
\begin{equation}\label{eq:defq}
	\Ccal^0 (C, f) \ni h\longmapsto \Qcal_j (h) = \int_{C_j} (h-f) \theta \in \C^3,\quad j =1,\ldots,n.
\end{equation}
Consider the map $f = (f_1, f_2, f_3) = d X / \theta \in \Oscr_{\infty} (S|E, \mathbf{A}_*)$ and fix a meromorphic function $g$ on $M$
such that $g$ is holomorphic and nowhere vanishing on $M\setminus E$ and so is the product $gf$ on $S$, hence
\begin{equation}\label{eq:f0}
	f_0 := g f \in \Oscr (S, \mathbf{A}_*);
\end{equation}
note that $\mathbf{A}_*$ is conical. Observe that $g\in\Oscr_\infty(M|E)$ must have effective poles (respectively, zeros) at those points of $E$ where the meromorphic extension of $f$ to $S$ has a zero (respectively, pole), and of the same order, in order to make sure that $f_0$ has the range in $\mathbf{A}_*\subset\C^3\setminus\{0\}$.
Such a function $g$ is provided by the classical Weierstrass-Florack interpolation theorem (see \cite{Florack1948} or, e.g., \cite[Theorem 1.12.13]{AlarconForstnericLopez2021Book}).
Following \cite[Proof of Theorem 3.1]{AlarconLopez2019}, there are complete  holomorphic vector fields $V_1, \hdots , V_m$ on $\C^3$ ($m \ge 3$) tangential to $\mathbf{A}_*$ along $\mathbf{A}_*$, and functions $h_{i,j} \in (\Oscr (M)^m)^n$, $i =1,\hdots, m$, $j = 1, \hdots n$, such that
\begin{equation}\label{eq:hij}
	\ord_p (h_{i,j} ) \ge 2  k \quad \, \text{for every }  p \in \Lambda\cup E,\,  i =1,\hdots, m, \, j = 1, \hdots n.
\end{equation}
Denote by $\phi^i_t$ the flow of $V_i$ over $\mathbf{A}_*$ ($t \in \C$) and  set $\Phi \colon (\C^m)^n \times M \times \mathbf{A}_* \to \mathbf{A}_*$ to be the holomorphic map
\[
	\Phi (\zeta, p, z) = (\phi^1_{ \zeta_{1,1} h_{1,1} (p) } \circ \hdots  \circ \phi^m_{\zeta_{m,1} h_{m,1} (p)}\circ \hdots \circ \phi^1_{\zeta_{1,n} h_{1,n} (p) }  \circ  \hdots \circ \phi^m_{\zeta_{m,n} h_{m,n} (p)}) (z),
\]
where $\zeta= ((\zeta_{i,j})_{i=1,\hdots,m})_{j=1,\hdots, n} \in (\C^m)^n$. 
In view of \eqref{eq:f0} we may consider the holomorphic spray
\begin{equation}\label{eq:sprayf0}
	\Phi_0 \colon (\C^m)^n \times S   \to \mathbf{A}_*, \quad	\Phi_{0} (\zeta, p) = \Phi(\zeta, p , f_0 (p)).
\end{equation}
We can choose the vector fields $V_i$ and the functions $h_{i,j}$ so that the following properties are satisfied (see \cite[Proof of Theorem 3.1]{AlarconLopez2019}).
\begin{enumerate}[label=(\Roman*)]
\item \label{coref} $\Phi_{0} ( 0, \cdot ) = f_0$  (we say that $f_0$ is the {\em core} of the spray $\Phi_0$).
\item \label{intf}$f_0 -  \Phi_{0}(\zeta, \cdot) $ vanishes to order $k$ everywhere on $\Lambda\cup E$ for every $\zeta \in (\C^m)^n$.
\item \label{eq::domination} The spray $\Phi_{0}$ is $\Qcal^*$-dominating at $\zeta=0$, in the sense that the map
$\Qcal^* :(\C^m)^n \to (\C^3)^n$ given by
\[
	\Qcal^* (\zeta) = \Qcal \Big(  \frac{\Phi_{0} (\zeta, \cdot )}{g}    \Big)
\]
is a submersion at $\zeta = 0$.
\newcounter{Mayus4}\setcounter{Mayus4}{\value{enumi}}
\end{enumerate}
Concerning \ref{eq::domination}, see \eqref{eq:defq} and note that  \eqref{eq:f0}, \ref{intf}, and \eqref{eq:k>} guarantee that $\Phi_0 (\zeta, \cdot)/g  \in \Cscr^0 (C,f)$ for all $\zeta \in (\C^m)^n$, so $\Qcal^*$ is well defined.
		
\noindent {\em Step 2: Approximation of the core.} We shall now approximate $f$ by a map $\hat f=(\hat f_1,\hat f_2, \hat f_3)\in\Oscr_{\infty}(L | E,\mathbf{A}_*)$  ensuring that $\hat f_3 \theta = d \wh X_3$ on $L \setminus E$ for a function $\wh X_3\in \Oscr_{\infty} ( L | E )$ such that $\wh X_3 - X_3$ extends to $S$ as a continuous map vanishing to order $k$ on $\Lambda\cup E$, $\wh X_3 - X_3$ is as close as desired to $0$ on $S$, and $\wh X_3^{-1}(0) = X_3^{-1}(0)$.
For this, call 
\[
	A=f_3^{-1}(0)\subset S.
\] 
Since $X$ is nonflat, we have that $f_3$ is not identically zero and so $A$ is finite. Up to passing to larger $S$ and $k$ if necessary, we may assume that $A\subset \mathring S$ and
\begin{equation}\label{eq:k>>}
	k>\sup\{\ord_p (f_3): p\in A\};
\end{equation}	
cf.\ \eqref{eq:k>}.
Since $\c^*$ is an Oka manifold (see \cite[Def.\ 1.2]{Forstneric2009CR} and \cite[Corollary 5.6.4]{Forstneric2017E})
the Runge approximation theorem with jet interpolation  for maps into Oka manifolds (see \cite[Theorem 5.4.4]{Forstneric2017E} or, e.g., \cite[Theorem 1.13.3]{AlarconForstnericLopez2021Book}) gives a holomorphic function $\wh X_3 \colon L  \setminus E \to \C$ such that:
\begin{enumerate}[label= \rm (\roman*)]
\item \label{whX3a} $\wh X_3 - X_3$ extends continuously to $S$ and $\| \wh X_3  - X_3 \|_{r,S} \approx 0$, meaning that the norm is sufficiently small according to the needs in the subsequent argumentation.
\item \label{whX3i} $\wh X_3- X_3$ vanishes to order $k$ everywhere on $\Lambda\cup E\cup A \cup\{p_0\}$. 
\item \label{whX3z} The divisor of $\wh X_3$ on $L\setminus E$ equals that of $X_3$ on $S\setminus E$.
\newcounter{Mayus3}\setcounter{Mayus3}{\value{enumi}}
\end{enumerate}
It follows that $\wh X_3$ is meromorphic on $L$ and its divisor on $L$ equals the one of $X_3$ on $S$.
Let us explain the above approximation more carefully. By  Weierstrass interpolation theorem there exists $\tau \in \Oscr_{\infty} (L)$ whose divisor on $L$ equals that of $X_3$ on $S$. Then, the function $X_3/\tau$ is holomorphic and nowhere vanishing on $S$.  By \cite[Theorem 1.13.3]{AlarconForstnericLopez2021Book}, we may approximate $X_3/\tau\colon S \to\C^*$ uniformly on $S$ by a holomorphic function $\sigma\colon L\to\C^*$ which agrees with $X_3/\tau$ to a high enough order everywhere on $\Lambda\cup E\cup A \cup \{p_0\}$. The function $\wh X_3=\sigma\tau\colon L\to\C$ then clearly satisfies the first part of \ref{whX3a}, as well as \ref{whX3i} and \ref{whX3z}. By the maximum modulus principle, we have that
\[
	\| \wh X_3  - X_3 \|_{0, S} = \| \wh X_3  - X_3 \|_{0, bS} = \big\| \tau ( \sigma - \frac{X_3}{\tau} ) \big\|_{0,b S} \le   \| \tau \|_{0, b S} \big\|  \sigma - \frac{X_3}{\tau} \big\|_{0,b S}.
\]
By the choice of $\tau$, the function $\tau|_{ b S} \colon bS \to \C$ is nonvanishing and analytic, so its $\Cscr^0$-maximum norm is well defined and non zero. Thus, since $\wh X_3 - X_3$ is holomorphic on $S$, the second part of \ref{whX3a} holds provided that $\sigma$ is close enough to $X_3 / \tau$ on $S$.
		
Define 
\begin{equation}\label{eq:f3}
	\hat  f_3 := d \wh X_3/\theta \in\Oscr_{\infty}(L | E).
\end{equation}
Properties \ref{whX3a}, \ref{whX3i}, and
\eqref{eq:k>>} together with Hurwitz theorem guarantee that: 
\begin{enumerate}[label= \rm (\roman*)]
\setcounter{enumi}{\value{Mayus3}}
\item \label{whf3a} $\hat f_3 - f_3$ extends continuously to $S$ and $\| \hat f_3  - f_3 \|_{r-1, S} \approx 0$.
\item \label{whf3i} $\hat f_3-f_3$ vanishes to order $k-1$ everywhere on $\Lambda\cup E\cup A $. 
\item \label{whf3z} The divisor of $\hat f_3$ on $S$ equals that of $f_3$.
\setcounter{Mayus3}{\value{enumi}}
\end{enumerate}		
Set $\eta := f_1 + \igot f_2 \in \Oscr_{\infty}(S| E)$. Since the map $f=(f_1,f_2,f_3)$ has range in $\mathbf{A}_*$ \eqref{eq:nullquadric}, it turns out that
\begin{equation}\label{eq:f1f2}
	\begin{split}
		\frac{f_3^2}{\eta}=-f_1+\igot f_2\in\Oscr_{\infty}(S | E ),\quad \quad \quad \quad \quad\quad\quad\quad \\
		f_1 =   \frac{1}{2} \Big( \eta - \frac{f_3^2}{\eta }\Big),
		\quad\text{and}\quad
		f_2 = -\frac{\igot}{2} \Big( \eta + \frac{f_3^2}{\eta } \Big) \quad \text{on } S \setminus E, 
	\end{split}
\end{equation}
and $\eta$ and $f_3^2/\eta$ have no common zeros on $S \setminus E$. 
Arguing as above, \cite[Theorem 1.13.3]{AlarconForstnericLopez2021Book} provides a function $\hat \eta\in\Oscr_{\infty}(L|E)$ such that:
\begin{enumerate}[label= \rm (\roman*)]
\setcounter{enumi}{\value{Mayus3}}
\item \label{whea} $\hat \eta - \eta$ extends continuously to $S$ and $\| \hat \eta-\eta \|_{r-1, S}\approx 0$.
\item \label{whei} $\hat \eta-\eta$ vanishes to order $k-1$ at every point in $\Lambda\cup E\cup A$.
\item \label{whez} The divisor of $\hat \eta$ on $L$ agrees with that of $\eta$ on $S$.
\setcounter{Mayus3}{\value{enumi}}	
\end{enumerate}
By \eqref{eq:f3}, \ref{whf3z},  \eqref{eq:f1f2}, and \ref{whez}, the function $\hat f_3^2/\hat \eta$ is holomorphic on $L\setminus E$, and hence so are
\[
	\hat f_1 :=   \frac{1}{2} \Big( \hat \eta - \frac{ \hat f_3^2}{ \hat \eta }\Big) \quad \text{and}\quad \hat f_2 := -\frac{\igot}{2} \Big( \hat \eta + \frac{\hat f_3^2}{\hat \eta } \Big). 
\]
In view of \ref{whf3z} and \ref{whez}, $\hat f_3^2/\hat \eta$ and $f_3^2/\eta$ have the same divisor on $S$, thus $\hat \eta$ and $\hat f_3^2/\hat \eta$ have no common zeros on $L\setminus E$ (recall that $\eta$ and $f_3^2/\eta$ have no common zeros on $S\setminus E$ and $\hat \eta$ is holomorphic and vanishes nowhere on $L\setminus S$). This shows that  $\hat f:=(\hat f_1, \hat f_2,\hat f_3) \in \Oscr_{\infty} (L | E, \mathbf{A}_*)$.
Moreover, it is easily seen that: 
\begin{enumerate}[label= \rm (\roman*)]
\setcounter{enumi}{\value{Mayus3}}
\item \label{whfa}$ \hat f - f$ extends continuously to $S$ and  $\| \hat f - f \|_{r-1, S} \approx 0 $ on $S$.
\item \label{whfi} $\hat f-f$ vanishes to order $k-1$ at every point in $\Lambda\cup E$.
\setcounter{Mayus3}{\value{enumi}}
\end{enumerate}

\noindent {\em Step 3: Correction of periods.} By \ref{whX3a}, \ref{whX3i}, and \eqref{eq:f3}, we have that 
\[
	\int_{C_j} (\hat f_3 - f_3)\theta=  0\quad \text{for all $j=1,\ldots,n$}. 
\]
However, the analogous conditions for $\hat f_1$ and $\hat f_2$ need not hold. To arrange this, set
\begin{equation}\label{eq:hf0}
	\hat f_0 := g \hat f  \in \Oscr (L , \mathbf{A}_*),
\end{equation}
see \eqref{eq:f0}, \ref{whfa}, \ref{whfi}, and \eqref{eq:k>}.  Since $\hat f_0$ has range in $\mathbf{A}_*$, the map
\[
	\wh \Phi_{0} \colon (\C^m)^n \times L \to \mathbf{A}_*, \quad	\wh\Phi_{0} (\zeta, p )  = \Phi (\zeta, p , \hat f_0(p)),
\]
is a holomorphic spray (cf.\ \eqref{eq:sprayf0}) such that:	
\begin{enumerate}[label= \rm (\Roman*)]
\setcounter{enumi}{\value{Mayus4}}
\item \label{corewhf} $\wh \Phi_{0} (\cdot , 0) = \hat f_0$ (that is, $\hat f_0$ is the core of the spray).
\item \label{intwhf} $\hat f_0 - \wh \Phi_{0} ( \zeta, \cdot )$ vanishes to order $2k-1$ everywhere on $\Lambda \cup E$ for every $\zeta \in (\C^m)^n$; see \eqref{eq:hij} and \cite[Lemma 2.2]{AlarconCastro-Infantes2019APDE}.
\setcounter{Mayus4}{\value{enumi}}
\end{enumerate}
On the other hand \eqref{eq:hf0}, \ref{intwhf}, and \eqref{eq:k>} ensure that 
\begin{equation}\label{eq:intspray}
	\ord_p \Big( \frac{\wh \Phi_{0} (\zeta, \cdot )}{g}   - \hat f  \Big) \ge k -1>0 \quad \text{for every } \zeta \in (\C^m)^n,\;  p \in \Lambda \cup E.
\end{equation}
In particular, $\wh \Phi_{0} (\zeta, \cdot ) / g-\hat f$ is holomorphic on $L$.
Taking also \ref{whfa} into account, we have that $\wh \Phi_{0} (\zeta, \cdot ) / g \in \Cscr^0 (C, f)$, and so we are entitled to define the period map $ \wh \Qcal :(\C^m)^n  \to (\C^3)^n$ given by
\[
	\wh \Qcal (\zeta) = \Qcal \Big(  \frac{\wh \Phi_{0} (\zeta, \cdot )}{g}    \Big),  \quad \zeta\in (\C^m)^n;
\]
see \eqref{eq:defq} and \ref{eq::domination}.
Properties \ref{coref}, \ref{eq::domination}, \ref{whfa}, and \ref{corewhf} ensure that $\wh \Qcal$ and $\Qcal^*$ are so close that $\wh \Qcal \colon U \to \wh \Qcal (U) \subset (\C^3)^n$ is a submersion on a sufficiently small compact neighborhood $U$ of $0 \in (\C^m)^n$.  Hence the implicit function theorem provides a parameter $\zeta_0 \in U\subset (\C^m)^n$ such that $\wh \Qcal(\zeta_0)= \Qcal^* (0) =  \Qcal(f) = 0$; see \eqref{eq:defq}. Thus,
\[
	\tilde f := \frac{\wh \Phi_{0} (\zeta_0, \cdot )}{g}  \in \Oscr_{\infty} (L | E, \mathbf{A}_*)
\]
satisfies
\begin{equation}\label{eq:(D)}
	\int_{C_j} (\tilde f - f)\theta = 0\quad \text{for every $j =1, \hdots, n$}
\end{equation} 
and the following properties:
\begin{enumerate}[label=(\Alph*)]
\item \label{extend} $\tilde f - \hat f$ extend continuously to $L$; see \eqref{eq:intspray}.
\item \label{aprox} $\| \tilde f - \hat f\|_{r-1, L}  \approx 0$; see \eqref{eq:hf0} and \ref{corewhf}.
\item   \label{int}  $\ord_p (\tilde f - \hat f ) \ge k-1$ for every $p \in \Lambda \cup E$; see \eqref{eq:intspray}.
\end{enumerate}
Since $f \theta=dX$ is exact on $S \setminus E$ and $S$ is a strong deformation retract of $L$, it follows from the definition of $\{C_j\}_{j=1}^n$ in Step 1, \ref{extend}, \ref{whfa}, and \eqref{eq:(D)}, that $\tilde f \theta$ is exact on $L \setminus E$. Thus we may define the holomorphic null curve $\wt X = (\wt X_1, \wt X_2, \wt X_3) \colon L \setminus E \to \C^3$ given by 
\[
	\wt X (p)  = X (p_0) + \int_{p_0}^{p} \tilde f \theta, \quad p \in L \setminus E.
\]
It is clear in view of \ref{extend}, \ref{aprox}, \ref{int}, \ref{whfa}, \ref{whfi}, and \eqref{eq:(D)} that $\wt X$ is nonflat  and satisfies \ref{conda}, \ref{conda1}, and \ref{condb}, whenever the approximations in  \ref{whfa}  and \ref{aprox} are close enough.

To finish, let us check \ref{condc}. 
By \eqref{eq:EX3Lambda}, \ref{whX3i}, \ref{int}, \ref{whX3z}, and \eqref{eq:(D)}, we have that $\wt X_3 - \wh X_3$ extends continuously to $L$ and vanishes to order $k$ everywhere on $X_3^{-1}(0) \cup \{ p_0\}=\wh X_3^{-1} (0)\cup\{p_0\}$, so $\wh X_3^{-1}(0) \subset  \wt X_3^{-1}(0)$. Moreover, condition \ref{aprox} shows that
$\| \wt X_3 - \wh X_3  \|_{r, L} \approx 0$.
By \eqref{eq:EX3Lambda} and \ref{whX3z}, $\wh X_3^{-1} (0) \subset \Lambda \subset \mathring S$, so Hurwitz theorem and  \eqref{eq:k>} imply that $\wh X_3^{-1}(0) = \wt X_3^{-1}(0)$ provided that $\| \wt X_3 - \wh X_3  \|_{r, L}$ is small enough.
Property \ref{condc} then follows by \ref{whX3z}.
This completes the proof.	
\end{proof}	

Lemma \ref{lemma:ind} is proved.
\end{proof}

By what has been done in Sections \ref{sec:lindf} and \ref{sec:lproper}, Lemma \ref{lemma:indf} and hence Lemma \ref{lemma:proper} are now proved. This completes the proof of Theorem \ref{th:mt} in Section \ref{sec:ms}.

%
%

\section{$\CMC$ surfaces and faces}\label{sec:bryant}
\noindent
In this section we establish approximation results with interpolation for both Bryant surfaces in $\h^3$ and $\CMC$ faces in $\S^3_1$, extending Corollaries \ref{co:bryant} and \ref{co:deSitter}.

Let us first explain with more detail the canonical projection $\pi_H:  \SL \to \h^3$; see \eqref{eq:pih} below. Denote by $\mathbb{L}^4$ the Minkowski space of dimension $4$ with the canonical Lorentzian metric of signature $(-+++)$,
and consider the hyperboloid model for the hyperbolic $3$-space
\[
	\h^3 = \left\{ (x_0, x_1, x_2, x_3) \in 	\mathbb{L}^4 : - x_0^2 +  x_1^2 + x_2^2 + x_3^2 = -1, x_0  > 0  \right\}
\]
with metric induced by $\mathbb{L}^4$. 
Identify $\mathbb{L}^4$ with the space of hermitian matrices ${\rm Her}(2)\subset \Mcal_2(\C)$ by
\[
	\mathbb{L}^4 \ni (x_0, x_1, x_2, x_3) \longleftrightarrow  \left(  \begin{array}{cc}
		x_0 + x_3 & x_1 + \igot x_2 \\
		x_1 - \igot x_2 & x_0- x_3
	\end{array} \right) \in {\rm Her}(2).
\]
Under this identification it turns out that
\[
	\h^3 = \{ A \overline{A}^T : A \in \SL \} =  \SL / \rm{SU}_2,
\]
where $\overline \cdot$ and $\cdot^T$ mean complex conjugation and transpose matrix, respectively, and $ \rm{SU}_2 \subset \SL$ is the set of unitary matrices.
Recall that the projection
\begin{equation}\label{eq:pih}
	\pi_H \colon  \SL \to \h^3,\quad \pi_H (A) =  A\, \overline{A}^T,
\end{equation}
is proper and takes holomorphic null curves in $\SL$ to conformal $\CMC$ immersions in $\h^3$. 
Conversely, every simply connected Bryant surface in $\h^3$ lifts to a holomorphic null curve in $\SL$.

The following extension of Corollary \ref{co:bryant} is the first main result of this section.
%
%
\begin{theorem}\label{th:Bryant} 
Let $M$ be an open Riemann surface, $K \subset M$ a locally
connected closed subset all whose connected components are smoothly bounded simply connected compact domains, $\Lambda$ and $E$ a pair of disjoint closed discrete subsets of $M$ with $\Lambda \subset \mathring K$ and $E \cap K = \varnothing$, and $k:\Lambda\to\n$ and $m: E \to\n$ maps. Then,
every conformal $\CMC$ immersion $\varphi: K \to\h^3$ can be approximated uniformly on $K$ by almost proper (hence complete) conformal $\CMC$ immersions $\psi: M\setminus E\to\h^3$ such that:
\begin{enumerate}[label= \rm (\roman*)]
\item \label{prop1} The immersions $\psi$ and $\varphi$ have a contact of order $k(p)$ at every point $p\in\Lambda$.
\item \label{prop2} For each $p\in E$, the end of $\psi$ corresponding to $p$ is of finite total curvature, regular, and of multiplicity $m(p)$; it is in addition smooth if $m(p)=1$. 
\item\label{prop3} $\psi$ lifts to a holomorphic null curve $M\setminus E\to\SL$ of class $\Oscr_\infty(M|E,\SL)$ with an effective pole at each point of $E$.
	\end{enumerate}
Moreover, if $\varphi  : K \to \h^3$ is proper then the approximation can be done by proper conformal $\CMC$ immersions $\psi:M\setminus E$ with these properties. 
\end{theorem}

This result improves \cite[Theorem 1.1]{AlarconCastro-InfantesHidalgo2023} by adding uniform approximation and ensuring the properness of $\psi$ under the necessary compatibility condition. Furthermore, it allows to prescribe any given multiplicity at the points of $E$; note that in \cite[Theorem 1.1]{AlarconCastro-InfantesHidalgo2023} the given map $m$ is required to take only odd values.

Before proving Theorem \ref{th:Bryant}, let us briefly explain some of the notions in its statement.
Let $\varphi: M \to \h^3$ be a conformal $\CMC$ immersion, where $M$ is an open Riemann surface.  Bryant defined its {\em hyperbolic Gauss map} $G : M \to \partial_{\infty} \h^3$, where $\partial_\infty \h^3$ is the ideal boundary of $\h^3$, by setting  $G(p)$ to be the intersection of the oriented normal geodesic which passes through $\varphi(p)$ with $\partial_\infty \h^3$, $p \in M$. One can identify $\partial_\infty \h^3$ with $\cp^1$ and it turns out that $\CMC$ surfaces in $\h^3$ are characterized by the holomorphicity of their  hyperbolic Gauss map \cite{Bryant1987Asterisque}. 
Complete $\CMC$ ends of finite total curvature are conformally equivalent to a punctured disc \cite{Huber1957} but the hyperbolic Gauss map need not extend to the puncture. 
Complete $\CMC$ ends where the hyperbolic Gauss map extends meromorphically are called {\em regular} \cite[Definition 4.1]{UmeharaYamada1993AM}.
The hyperbolic Gauss map is fundamental to understand the geometry of $\CMC$ surfaces in $\h^3$ as it is shown in \cite{UmeharaYamada1993AM}, among others. In particular, Umehara and Yamada defined in \cite[Sec.\ 5]{UmeharaYamada1993AM} the {\em multiplicity} $m \in \n$ of a regular end and proved that it is embedded if and only if it is of multiplicity $1$ \cite[Theorem 5.2]{UmeharaYamada1993AM}.  Thus, roughly speaking, the multiplicity measures how far is the regular end to be embedded. It is also seen that an immersed $\CMC$ regular end is always tangent to $\partial_\infty \h^3$ \cite[Remark 5.5]{UmeharaYamada1993AM};
it is said to be {\em smooth} if  the immersion extends smoothly to the puncture through $\partial_\infty \h^3$ \cite[p.\ 590]{BohlePeters2009}. Clearly, smooth $\CMC$ ends are properly embedded (hence of multiplicity one), regular, and of finite total curvature.
 More generally, properly embedded $\CMC$ annular ends are regular and of finite total curvature by \cite[Theorem 10]{CollinHauswirthRosenberg2001AM}.
We refer to the surveys \cite{JensenMussoNicolodi2016, Rossman01, Rosenberg2002} on $\CMC$ surfaces in $\h^3$; see also \cite[Sec.\ 1]{AlarconCastro-InfantesHidalgo2023} and the references therein. 

Set $\D=\{z\in\C\colon |z|<1\}$ and $\D^*=\D\setminus\{0\}$. We begin with the following.
%
%
\begin{lemma}\label{lemma:m}
Let $F=(F_1, F_2, F_3, F_4) : \D^*\to \SL$ be a holomorphic null curve extending meromorphically to $\D$ with an effective pole at the origin and assume that the conformal $\CMC$ immersion
$\varphi=\pi_H\circ F : \D^* \to \h^3$ (see \eqref{eq:pih}) is not totally umbilic and its secondary Gauss map extends holomorphically to the origin (see \cite[Def.\ 1.2]{UmeharaYamada1996Ann}). Then the following hold.
\begin{enumerate}[label= \rm (\roman*)]
\item\label{item-lemma-i} The end of $\varphi$ at the origin is proper, of finite total curvature, regular, and of multiplicity $ \left|	\ord_0  (( F_3/ F_1 ) ' ) +1 \right|$; see \eqref{eq:orden}.
\item\label{item-lemma-ii} If the end of $\varphi$ at the origin is embedded, then it is a smooth end.
\end{enumerate}
\end{lemma}
%
%
\begin{proof}
Since $F$ has an effective pole, it is clear that $\varphi$ has a proper (hence complete) 
end at the origin. Its {\em hyperbolic Gauss map} $G : \D^* \to \partial_\infty \h^3\equiv\cp^1$ is given by
\[
	G= \frac{d F_1}{d F_3} = \frac{d F_2}{d F_4}
\]
(see \cite{Bryant1987Asterisque} or \cite[Eq. (1.12)]{UmeharaYamada1993AM}),
so $\varphi$ has a regular end at $0$.
Write 
\[
	F^{-1}  d F = \left( \begin{array}{cc}
		g &  -g^2\\
		1& -g
	\end{array}\right) \omega,
\]
where
\begin{equation}\label{eq:secondary}
	g= -\frac{dF_2}{dF_1} = - \frac{d F_4}{dF_3}
\end{equation}
is a meromorphic function on $\D$
(the {\em secondary Gauss map} of $\varphi$; see \cite[Def.\ 1.2]{UmeharaYamada1996Ann})
which may be assumed to be holomorphic on $\D$ (recall that $g$ extends holomorphically to $0$ by hypothesis),
and 
\[
	\omega = F_1 d F_3 - F_3 d F_1
\]
is a meromorphic $1$-form on $\D$ which is holomorphic on $\D^*$. 
A straightforward computation shows that  the \emph{Hopf differential} $Q = \omega d g$ of $\varphi$ (see \cite{Bryant1987Asterisque} 
or \cite[Eq.\ (1.10)]{UmeharaYamada1993AM}) is the holomorphic $2$-form
\[
	Q  = \left(\frac{F_1''}{F_1} - \frac{F_1'}{F_1} \cdot \frac{F_3''  F_1 - F_3 F_1''}{F_3 ' F_1 - F_3 F_1 '}  \right) dz^2 \quad \text{on } \D^*
\]
(cf.\ \cite[Eq.\ (6.12)]{AlarconCastro-InfantesHidalgo2023}).
Note that $Q$ extends meromorphically to the origin with $\ord_ 0 (Q) \ge -2$ (see \eqref{eq:orden}), and so its Laurent expansion reads
\[
	Q =  \Big(    \sum_{j= -2}^{\infty}  \hat q_j z^j  \Big) dz^2\quad \text{on $\D^*$},
\]
where $\hat q_j \in \C$ for every $j\ge -2$.
(This also follows from \cite[Prop.\ 6]{Bryant1987Asterisque} or \cite[Lemma 1.5]{UmeharaYamada1993AM}.) 
Moreover, the Riemannian metric induced on $\D^*$ by the hyperbolic metric in $\h^3$ via $\varphi$ is
$(1 + |g|^2 )^2 |\omega |^2$
(see \cite[Eq.\ (1.9)]{UmeharaYamada1993AM}), and so the end of $\varphi$ at $0$ is of finite total curvature. 
Since $\varphi$ is not totally umbilic (that is, $Q\not\equiv 0$) and its secondary Gauss map is holomorphic at the origin,
the discussion in \cite[Sec. 5]{UmeharaYamada1993AM} ensures that the multiplicity of its end at the origin is
\[
	m= \sqrt{(\ord_0 (\omega) + 1)^2 + 4 \hat q_{-2}} \in \N.
\]
It is seen by direct computation that 
$\ord_0 (\omega)= l+ 2k$ and $\hat{q}_{-2} = - k (k+1+l)$, where
$k = \ord_0 (F_1) \in \Z$ and $l = \ord_0  \big( ( F_3/F_1 )'  \big) \in \Z$.
It thus follows that $m = | l +1|$, thereby proving \ref{item-lemma-i}.
Let us now check \ref{item-lemma-ii}, so assume that the end of $\varphi$ at the origin is embedded.
By \cite[Lemma 3.2 (i)]{BohlePeters2009}, we have that the end is smooth if and only if $\min \{ \ord_0 (( F ' F^{-1})_j)\colon  j= 1,2,3,4\}=-2$. The number in the left remains invariant when  replacing $F$ by either of the following holomorphic null curves:
\[
	\left(\begin{matrix}
		F_2 & -F_1\\
		F_4& -F_3
	\end{matrix}\right), 
	\quad 
	\left(\begin{matrix}
		 F_3 & F_4\\
		- F_1 & - F_2
	\end{matrix}\right) 
	, \quad 
	 \left(\begin{matrix}
		F_4  & - F_3\\
		- F_2 & F_1
	\end{matrix}\right).
\]
So, we may assume that $\min \{ \ord_0 (F_j) \colon j =1,2,3,4\} = \ord_0 (F_1)<0$. Therefore, the holomorphic null curve $X = (X_1, X_2, X_3) = \Tcal^{-1} \circ F  : \D^* \to \C^3$ extends holomorphically to $0$ with $X_3 (0) = 0$; see \eqref{eq:t-1}. In this situation,  \cite[Lemma 6.2]{AlarconCastro-InfantesHidalgo2023} guarantees that the embedded end of $\varphi=\pi_H\circ \Tcal\circ X$ at $0$ is smooth, as claimed. 
\end{proof}
%
%
\begin{proof}[Proof of Theorem \ref{th:Bryant}] 
We may assume that $\varphi$ is not totally umbilic; otherwise we add a point  to $\Lambda$ and extend $\varphi$ to a neighborhood of this point as a non-totally umbilic conformal $\CMC$ immersion.  We also assume that $\varphi  : K \to \h^3$ is proper; the proof is simpler otherwise. Lift $\varphi : K \to \h^3$ to a holomorphic null curve $F : K \to \SL$ such that $\pi_H\circ F = \varphi$ \eqref{eq:pih}, and note that $F$ is proper by properness of $\pi_H$ and $\varphi$. 
Let  $V= \bigcup_{p\in E} V_p$ be a union of holomorphic coordinate open discs $V_p \ni p$, for every $p \in E$, with pairwise disjoint closures and such that $\overline V \cap K = \varnothing$. Fix  $p \in E$ and let $z : V_p \to \D$ be a holomorphic coordinate on $V_p$ with $z(p) = 0$.     

Extend $F$ to $V_p \setminus \{p\}$ as follows.

\noindent {\em Case 1.} If  $m(p) = 1$, where $m$ is the map given in the statement, set
\[
	F (z)= \left( z^{-2}, \, -\frac{4}{3} z, \, z^{-1}, \, -\frac{1}{3} z^{2}  \right) \text{ for all }z\in V_p \setminus \{p\}.
\]

\noindent {\em Case 2.} If $m(p) >1$,  set 
\[
F(z)=\left(\frac{1}{z},\,\frac{-z^{m(p)+1}}{m(p)+2},\,\frac{-1}{m(p) z^{m(p)+1}},\,\frac{(m(p)+1)^2z}{(m(p)+1)^2 -1}\right)
\text{ for all }z\in V_p\setminus\{p\}.
\]

This gives a holomorphic null curve $F = (F_1, F_2, F_3, F_4)  : K \cup (V \setminus E) \to \SL$ extending meromorphically to $K \cup V$ with an effective pole at every point of $E$. Clearly $\ord_p (F_3) \not = \ord_p (F_1)$ for every $p\in E$ (see \eqref{eq:orden}), hence
\begin{equation}\label{eq:condm}
	\left|	\ord_p  \left( \left( F_3 / F_1 \right)  ' \right) +1 \right| = | \ord_p (F_3) - \ord_p (F_1)| = m(p) \text{ for every }p\in E,
\end{equation}
where the derivative is taken in the holomorphic coordinates $z$ on $V$.  Extend $\varphi$ to $V\setminus E$ as $\varphi=\pi_H\circ F$ (see  \eqref{eq:pih}), a conformal $\CMC$ immersion. Note that its secondary Gauss map extends holomorphically to $E$; see \eqref{eq:secondary}.
Now extend the given map $k : \Lambda \to \N$ to $E$ by setting 
\begin{equation}\label{eq:kbryant}
	k(p) = 1+ \max\{    |\ord_p (F_i) | \colon i =1,2,3,4\},\quad p\in E.
\end{equation}
Since $F : K \to \SL$ is proper and $F$ has an effective pole at every point of $E$, there is a set $W\subset M$ which is a locally finite union of pairwise disjoint smoothly bounded simply connected compact domains such that $K\cap W=\varnothing$, $E \subset \mathring  W \Subset V$, and $F|_{S \setminus E} : S \setminus E \to \SL$ is a proper map for $S = K \cup W$.  
Fix $\varepsilon_0 >0$ to be specified later. Theorem \ref{th:mt} applied to $F|_{S\setminus E}$ furnishes us with a proper nonflat holomorphic null curve $\wt F = (\wt F_1, \wt F_2, \wt F_3, \wt F_4) \colon M \setminus E \to \SL$ satisying the following.
\begin{enumerate}[label=(\Roman*)]
\item \label{Fpoles} $\wt F - F$ extends continuously to $S$.
\item \label{Faprox} $\| \wt F - F \|_{0,S}  < \varepsilon_0$.
\item \label{contact} $\wt F - F$ vanishes to order $k (p)$ at every point $p \in \Lambda \cup E$.
\end{enumerate}
Therefore, the composition
\[
	\psi := \pi_H \circ \wt F  \colon  M \setminus E \to \h^3
\]
is a proper conformal $\CMC$ immersion; see  \eqref{eq:pih}. We claim that $\psi$ satisfies the conclusion of the theorem provided that $\varepsilon_0>0$ is chosen small enough.
Indeed, \ref{Faprox} ensures that $\varphi$ and $\psi$ are as close as desired on $K$, 
and \ref{prop1} follows from \ref{contact}. 
Observe that if the approximation of $\varphi$ by $\psi$ on $K$ is good enough then $\psi$ is not totally umbilic. 
Moreover, $\wt F$ has an effective pole at every point of $E$ by \ref{Fpoles} since so has $F$.
Also note that the secondary Gauss map of $\psi$ extends holomorphically to the origin; take into account \eqref{eq:secondary}, \eqref{eq:kbryant}, and \ref{contact}.
Thus Lemma \ref{lemma:m} guarantees that the end of $\psi$ corresponding to each point $p\in E$ is of finite total curvature, regular, and of multiplicity $|	\ord_p  \big( ( \wt F_3 / \wt F_1 )  ' \big) +1 |$; it is in addition smooth if it is of multiplicity $1$.  By    \eqref{eq:kbryant} and \ref{contact}, we have that 
\[
	\ord_p (\wt F_3)= \ord_p (F_3) \not = \ord_p (F_1) = \ord_0 (\wt F_1)\quad  \text{for every } p\in E,
\]
and hence, using also \eqref{eq:condm}, we obtain that
\[
	\big|	\ord_p  \big( ( \wt F_3 / \wt F_1 )  ' \big) +1 \big| = |\ord_p (\wt F_3) - \ord_p (\wt F_1)|=  m (p) \quad \text{for every } p\in E.
\]
This shows \ref{prop2}. Finally, condition \ref{prop3} is obvious from the definition of $\psi$.
\end{proof}

The second main result in this section is an analogue of Theorem  \ref{th:Bryant} for $\CMC$ faces in de Sitter $3$-space $\S^3_1$. Recall that 
\[
	\S_1^3 = \{   ( x_0, x_1, x_2, x_3) \in \mathbb{L}^4 \colon -x_0^2 + x_1^2 + x_2^2 + x_3^2 = 1 \}
\] 
with metric induced from $\mathbb{L}^4$. It can be identified as $\S_1^3=\SL/ {\rm SU_{1,1}}$, where 
\begin{equation}\label{eq:SU11}
	{\rm SU_{1,1}} = \Big\{ A \in \SL \colon  A \left(\begin{array}{cc}
		1&0\\
		0&-1
	\end{array} \right)   \overline{A}^T = {\rm Id} \Big\}
\end{equation}
(see \cite[\textsection 0]{FRUYY09}), and the canonical projection 
\[
	\pi_S  \colon \SL \to \S_1^3, \quad \pi_S (A) = A \left(\begin{array}{cc}
		1&0\\
		0&-1
	\end{array} \right)   \overline{A}^T
\]
takes holomorphic null curves $M\to\SL$ into \emph{$\CMC$ faces} $M\to\S_1^3$ provided that its secondary Gauss map (see \cite[Remark 1.2]{Fujimori06}) does not have constant norm equal to $1$ \cite[Theorem 1.9]{Fujimori06}. Moreover, if such a holomorphic null curve is complete then it projects to a weakly complete $\CMC$ face in $\S_1^3$. Conversely, every simply connected $\CMC$ face $M\to\S_1^3$ lifts to a holomorphic null curve $M\to\SL$ such that the norm of its secondary Gauss map is not identically $1$.

The following extension of Corollary \ref{co:deSitter} can be proved by the same argument of the proof of Theorem \ref{th:Bryant}. 
%
%
\begin{theorem}\label{th:deSitter}
Let $M$, $K$, $\Lambda$, $E$, and $k$ be as in Theorem \ref{th:Bryant}, and let $E_1 \subset E$ be a subset. Then, every $\CMC$ face
$\phi: K \to\s_1^3$ can be approximated uniformly on $K$ by weakly complete $\CMC$ faces $\psi: M \setminus E \to\s_1^3$ such that:
\begin{enumerate}[label= \rm (\roman*)]
	\item \label{prop1s} $\psi$ and $\phi$ have a contact of order $k(p)$ at every point $p\in\Lambda$.
	\item \label{prop2s} For each $p\in E$, the end of $\psi$ corresponding to $p$ is complete, regular, and of finite total curvature (see \cite[Def.\ 2.1 and 2.8]{Fujimori06}); it is in addition embedded if and only if $p\in E_1$.
	\item\label{prop3s} $\psi$ lifts to a holomorphic null curve $M\setminus E\to\SL$ of class $\Oscr_\infty(M|E,\SL)$ with an effective pole at each point of $E$.
\end{enumerate}
\end{theorem}
Let us briefly explain how to grant \ref{prop2s}. For this, assume that $F : \D^* \to \SL$ is a holomorphic null curve as in the hypothesis of Lemma \ref{lemma:m} and with the norm of its secondary Gauss map not identicaly $1$. It follows that $\varphi = \pi_S \circ F : \D^* \to \s_1^3$ is a $\CMC$ face with a complete regular end of finite total curvature at the origin. Moreover, the end of $\varphi$ at the origin is embedded if and only if so is the end at the origin of the conformal $\CMC$ immersion $\pi_H \circ F:\D^*\to\h^3$; see \cite[Theorem 3.1 and Remark 3.2]{Fujimori06}. In view of this, \ref{prop2s} is ensured as in the proof of Theorem \ref{th:Bryant}. 

Note that Theorems \ref{th:Bryant} and \ref{th:deSitter} (even in case $E=\varnothing$)  do not follow by a recursive application of Corollaries \ref{co:bryant} and \ref{co:deSitter}, respectively, due to the assumption in the corollaries that $K$ is simply connected. So the former results are honestly more general than the latter.
A more general point of view that unifies $\CMC$ surfaces in $\h^3$ and $\CMC$ faces in $\S_1^3$ as projections of holomorphic null curves in $\SL$ is that of holomorphic null curves in the nonsingular complex hyperquadric $\mathbb{Q}_3\subset\CP^4$, viewed as the Grassmannian of Lagrangian 2-dimensional planes of $\C^4$ with its standard conformal structure. We refer to \cite{MussoNicolodi2022IJM,BohlePeters2009} for further information.

%
%
\subsection*{Acknowledgements}
Research partially supported by the State Research Agency (AEI) via the grants no.\ PID2020-117868GB-I00  and PID2023-150727NB-I00, and the ``Maria de Maeztu'' Unit of Excellence IMAG, reference CEX2020-001105-M, funded by MICIU/AEI/10.13039/501100011033 and ERDF/EU, Spain.

%
%

\begin{thebibliography}{10}

\bibitem{Abraham1963}
R.~Abraham.
\newblock Transversality in manifolds of mappings.
\newblock {\em Bull. Amer. Math. Soc.}, 69:470--474, 1963.

\bibitem{Alarcon2017JGA}
A.~Alarc\'on.
\newblock Proper holomorphic {L}egendrian curves in {$SL_2(\Bbb C)$}.
\newblock {\em J. Geom. Anal.}, 27(4):3013--3029, 2017.

\bibitem{AlarconCastro-Infantes2019APDE}
A.~Alarc{\'o}n and I.~Castro-Infantes.
\newblock Interpolation by conformal minimal surfaces and directed holomorphic
  curves.
\newblock {\em Anal. PDE}, 12(2):561--604, 2019.

\bibitem{AlarconCastro-InfantesHidalgo2023}
A.~Alarc{\'o}n, I.~Castro-Infantes, and J.~Hidalgo.
\newblock Complete {CMC}-1 surfaces in hyperbolic space with arbitrary complex
  structure.
\newblock {\em Commun. Contemp. Math.}, 27(1):Paper No. 2450011, 26, 2025.

\bibitem{AlarconForstneric2014IM}
A.~Alarc{\'o}n and F.~Forstneri\v{c}.
\newblock Null curves and directed immersions of open {R}iemann surfaces.
\newblock {\em Invent. Math.}, 196(3):733--771, 2014.

\bibitem{AlarconForstneric2015MA}
A.~Alarc{\'o}n and F.~Forstneri\v{c}.
\newblock The {C}alabi--{Y}au problem, null curves, and {B}ryant surfaces.
\newblock {\em Math. Ann.}, 363(3-4):913--951, 2015.

\bibitem{AlarconForstnericLopez2017CM}
A.~Alarc{\'o}n, F.~Forstneri\v{c}, and F.~J. L{\'o}pez.
\newblock Holomorphic {L}egendrian curves.
\newblock {\em Compos. Math.}, 153(9):1945--1986, 2017.

\bibitem{AlarconForstnericLopez2021Book}
A.~Alarc\'{o}n, F.~Forstneri\v{c}, and F.~J. L\'{o}pez.
\newblock {\em Minimal surfaces from a complex analytic viewpoint}.
\newblock Springer Monographs in Mathematics. Springer, Cham, [2021] \copyright
  2021.

\bibitem{AlarconLopez2012JDG}
A.~Alarc{\'o}n and F.~J. L{\'o}pez.
\newblock Minimal surfaces in {$\mathbb R^3$} properly projecting into
  {$\mathbb R^2$}.
\newblock {\em J. Differential Geom.}, 90(3):351--381, 2012.

\bibitem{AlarconLopez2013MA}
A.~Alarc{\'o}n and F.~J. L{\'o}pez.
\newblock Null curves in {$\mathbb {C}^3$} and {C}alabi-{Y}au conjectures.
\newblock {\em Math. Ann.}, 355(2):429--455, 2013.

\bibitem{AlarconLopez2019}
A.~Alarc\'{o}n and F.~J. L\'{o}pez.
\newblock Algebraic approximation and the {M}ittag-{L}effler theorem for
  minimal surfaces.
\newblock {\em Anal. PDE}, 15(3):859--890, 2022.

\bibitem{AlarconVrhovnik2024X}
A.~Alarc{\'o}n and T.~Vrhovnik.
\newblock {The Mittag-Leffler theorem for proper minimal surfaces and directed
  meromorphic curves}.
\newblock {\em arXiv e-prints}, 2024.

\bibitem{BehnkeStein1949}
H.~Behnke and K.~Stein.
\newblock Entwicklung analytischer {F}unktionen auf {R}iemannschen {F}l\"achen.
\newblock {\em Math. Ann.}, 120:430--461, 1949.

\bibitem{Bishop1958PJM}
E.~Bishop.
\newblock Subalgebras of functions on a {R}iemann surface.
\newblock {\em Pacific J. Math.}, 8:29--50, 1958.

\bibitem{BobenkoPavlyukevichSpringborn2003}
A.~I. Bobenko, T.~V. Pavlyukevich, and B.~A. Springborn.
\newblock Hyperbolic constant mean curvature one surfaces: spinor
  representation and trinoids in hypergeometric functions.
\newblock {\em Math. Z.}, 245(1):63--91, 2003.

\bibitem{BohlePeters2009}
C.~Bohle and G.~P. Peters.
\newblock Bryant surfaces with smooth ends.
\newblock {\em Comm. Anal. Geom.}, 17(4):587--619, 2009.

\bibitem{Bryant1987Asterisque}
R.~L. Bryant.
\newblock Surfaces of mean curvature one in hyperbolic space.
\newblock {\em Ast\'{e}risque}, 353 (1988)(154-155):12, 321--347, 1987.
\newblock Th\'{e}orie des vari\'{e}t\'{e}s minimales et applications
  (Palaiseau, 1983--1984).

\bibitem{Carleman1927}
T.~{Carleman}.
\newblock {Sur un th\'eor\`eme de Weierstra{\ss}.}
\newblock {\em {Ark. Mat. Astron. Fys.}}, 20(4):5, 1927.

\bibitem{Castro-InfantesChenoweth2020}
I.~Castro-Infantes and B.~Chenoweth.
\newblock Carleman approximation by conformal minimal immersions and directed
  holomorphic curves.
\newblock {\em J. Math. Anal. Appl.}, 484(2):123756, 2020.

\bibitem{Castro-InfantesHidalgo2024}
I.~Castro-Infantes and J.~Hidalgo.
\newblock C{MC}-1 surfaces in hyperbolic and de {S}itter spaces with {C}antor
  ends.
\newblock {\em Mediterr. J. Math.}, 21(5):Paper No. 167, 22, 2024.

\bibitem{CollinHauswirthRosenberg2001AM}
P.~Collin, L.~Hauswirth, and H.~Rosenberg.
\newblock The geometry of finite topology {B}ryant surfaces.
\newblock {\em Ann. of Math. (2)}, 153(3):623--659, 2001.

\bibitem{Conway1973}
J.~B. Conway.
\newblock {\em Functions of one complex variable}, volume~11 of {\em Graduate
  Texts in Mathematics}.
\newblock Springer-Verlag, New York-Heidelberg, 1973.

\bibitem{deLimaRoitman2002}
L.~L. de~Lima and P.~Roitman.
\newblock Constant mean curvature one surfaces in hyperbolic 3-space using the
  {B}ianchi-{C}al\`o{} method.
\newblock {\em An. Acad. Brasil. Ci\^enc.}, 74(1):19--24, 2002.

\bibitem{FarkasKra1992}
H.~M. Farkas and I.~Kra.
\newblock {\em Riemann surfaces}, volume~71 of {\em Graduate Texts in
  Mathematics}.
\newblock Springer-Verlag, New York, second edition, 1992.

\bibitem{Florack1948}
H.~Florack.
\newblock Regul\"are und meromorphe {F}unktionen auf nicht geschlossenen
  {R}iemannschen {F}l\"achen.
\newblock {\em Schr. Math. Inst. Univ. M\"unster,}, 1948(1):34, 1948.

\bibitem{FFW18}
J.~E. Forn{\ae}ss, F.~Forstneri\v{c}, and E.~F. Wold.
\newblock Holomorphic approximation: the legacy of {W}eierstrass, {R}unge,
  {O}ka-{W}eil, and {M}ergelyan.
\newblock In {\em Advancements in complex analysis---from theory to practice},
  pages 133--192. Springer, Cham, [2020] \copyright 2020.

\bibitem{Forstneric2009CR}
F.~Forstneri\v{c}.
\newblock Oka manifolds.
\newblock {\em C. R. Math. Acad. Sci. Paris}, 347(17-18):1017--1020, 2009.

\bibitem{Forstneric2017E}
F.~Forstneri\v{c}.
\newblock {\em Stein manifolds and holomorphic mappings (The homotopy principle
  in complex analysis)}, volume~56 of {\em Ergebnisse der Mathematik und ihrer
  Grenzgebiete. 3. Folge}.
\newblock Springer, Cham, second edition, 2017.

\bibitem{Fujimori06}
S.~Fujimori.
\newblock {Spacelike CMC 1 surfaces with elliptic ends in de Sitter 3-space}.
\newblock {\em Hokkaido Mathematical Journal}, 35(2):289 -- 320, 2006.

\bibitem{FKKRUY13}
S.~Fujimori, Y.~Kawakami, M.~Kokubu, W.~Rossman, M.~Umehara, and K.~Yamada.
\newblock Hyperbolic metrics on {R}iemann surfaces and space-like {CMC}-1
  surfaces in de {S}itter 3-space.
\newblock In {\em Recent trends in {L}orentzian geometry}, volume~26 of {\em
  Springer Proc. Math. Stat.}, pages 1--47. Springer, New York, 2013.

\bibitem{FRUYY09}
S.~Fujimori, W.~Rossman, M.~Umehara, K.~Yamada, and S.-D. Yang.
\newblock Spacelike mean curvature one surfaces in de {S}itter 3-space.
\newblock {\em Comm. Anal. Geom.}, 17(3):383--427, 2009.

\bibitem{GalvezMira2008}
J.~A. G\'alvez and P.~Mira.
\newblock The {L}awson correspondence for {B}ryant surfaces in explicit
  coordinates.
\newblock {\em Mat. Contemp.}, 35:61--72, 2008.

\bibitem{Huber1957}
A.~Huber.
\newblock On subharmonic functions and differential geometry in the large.
\newblock {\em Comment. Math. Helv.}, 32:13--72, 1957.

\bibitem{Hurwitz1895}
A.~Hurwitz.
\newblock Ueber die {B}edingungen, unter welchen eine {G}leichung nur {W}urzeln
  mit negativen reellen {T}heilen besitzt.
\newblock {\em Math. Ann.}, 46(2):273--284, 1895.

\bibitem{JensenMussoNicolodi2016}
G.~R. Jensen, E.~Musso, and L.~Nicolodi.
\newblock {\em Surfaces in classical geometries}.
\newblock Universitext. Springer, Cham, 2016.
\newblock A treatment by moving frames.

\bibitem{KokubuUmeharaYamada2003}
M.~Kokubu, M.~Umehara, and K.~Yamada.
\newblock An elementary proof of {S}mall's formula for null curves in {${\rm
  PSL}(2,{\bf C})$} and an analogue for {L}egendrian curves in {${\rm
  PSL}(2,{\bf C})$}.
\newblock {\em Osaka J. Math.}, 40(3):697--715, 2003.

\bibitem{Lawson1970AM}
H.~B. Lawson, Jr.
\newblock Complete minimal surfaces in {$S^{3}$}.
\newblock {\em Ann. of Math. (2)}, 92:335--374, 1970.

\bibitem{MartinUmeharaYamada2009CVPDE}
F.~Mart\'{i}n, M.~Umehara, and K.~Yamada.
\newblock Complete bounded null curves immersed in {$\mathbb{C}^3$} and {${\rm
  SL}(2,\mathbb{C})$}.
\newblock {\em Calc. Var. Partial Differential Equations}, 36(1):119--139,
  2009.
\newblock Erratum: Complete bounded null curves immersed in {$\mathbb{C}^3$}
  and {${\rm SL}(2,\mathbb{C})$}, Calc. Var. Partial Differential Equations
  46(1-2): 439--440 (2013).

\bibitem{Mergelyan1951}
S.~N. Mergelyan.
\newblock On the representation of functions by series of polynomials on closed
  sets.
\newblock {\em Doklady Akad. Nauk SSSR (N.S.)}, 78:405--408, 1951.

\bibitem{Mittag-Leffler1884AM}
G.~{Mittag-Leffler}.
\newblock {D\'emonstration du th\'eor\`eme de Laurent.}
\newblock {\em {Acta Math.}}, 4:80--88, 1884.

\bibitem{MussoNicolodi2009}
E.~Musso and L.~Nicolodi.
\newblock The spinor representation of {CMC} 1 surfaces in hyperbolic space.
\newblock {\em Note Mat.}, 28:317--339, 2009.

\bibitem{MussoNicolodi2022IJM}
E.~Musso and L.~Nicolodi.
\newblock Conformal geometry of isotropic curves in the complex quadric.
\newblock {\em Internat. J. Math.}, 33(8):Paper No. 2250054, 32, 2022.

\bibitem{Pirola2007AJM}
G.~P. Pirola.
\newblock Monodromy of constant mean curvature surface in hyperbolic space.
\newblock {\em Asian J. Math.}, 11(4):651--669, 2007.

\bibitem{Rosenberg2002}
H.~Rosenberg.
\newblock Bryant surfaces.
\newblock In {\em The global theory of minimal surfaces in flat spaces
  ({M}artina {F}ranca, 1999)}, volume 1775 of {\em Lecture Notes in Math.},
  pages 67--111. Springer, Berlin, 2002.

\bibitem{Rossman01}
W.~Rossman.
\newblock Mean curvature one surfaces in hyperbolic space, and their
  relationship to minimal surfaces in {E}uclidean space.
\newblock {\em J. Geom. Anal.}, 11(4):669--692, 2001.

\bibitem{Runge1885}
C.~Runge.
\newblock Zur {T}heorie der {E}indeutigen {A}nalytischen {F}unctionen.
\newblock {\em Acta Math.}, 6(1):229--244, 1885.

\bibitem{SaEarpToubiana2004}
R.~Sa~Earp and E.~Toubiana.
\newblock Meromorphic data for mean curvature one surfaces in hyperbolic
  three-space.
\newblock {\em Tohoku Math. J. (2)}, 56(1):27--64, 2004.

\bibitem{Small1994}
A.~J. Small.
\newblock Surfaces of constant mean curvature {$1$} in {${\bf H}^3$} and
  algebraic curves on a quadric.
\newblock {\em Proc. Amer. Math. Soc.}, 122(4):1211--1220, 1994.

\bibitem{Svetina2024}
A.~Svetina.
\newblock Approximation of holomorphic {L}egendrian curves with
  jet-interpolation.
\newblock {\em J. Math. Anal. Appl.}, 531(2):Paper No. 127839, 25, 2024.

\bibitem{UmeharaYamada1993AM}
M.~Umehara and K.~Yamada.
\newblock Complete surfaces of constant mean curvature {$1$} in the hyperbolic
  {$3$}-space.
\newblock {\em Ann. of Math. (2)}, 137(3):611--638, 1993.

\bibitem{UmeharaYamada1996Ann}
M.~Umehara and K.~Yamada.
\newblock Surfaces of constant mean curvature {$c$} in {$H^3(-c^2)$} with
  prescribed hyperbolic {G}auss map.
\newblock {\em Math. Ann.}, 304(2):203--224, 1996.

\bibitem{Weierstrass1885}
K.~{Weierstrass}.
\newblock {Ueber die analytische Darstellbarkeit sogenannter willk\"urlicher
  Functionen einer reellen Ver\"anderlichen.}
\newblock {\em {Berl. Ber.}}, 1885:633--640, 789--806, 1885.

\end{thebibliography}

%
%
%
\noindent Antonio Alarc\'{o}n, Jorge Hidalgo

\noindent Departamento de Geometr\'{\i}a y Topolog\'{\i}a e Instituto de Matem\'aticas (IMAG), Universidad de Granada, Campus de Fuentenueva s/n, E--18071 Granada, Spain.

\noindent  e-mail: {\tt alarcon@ugr.es, jorgehcal@ugr.es}

\end{document}